\documentclass[12pt]{article}
\usepackage{amsthm,amsmath}
\usepackage{natbib}

\usepackage{mathtools,amssymb,xparse,tikz-cd}
\usepackage{etoolbox}
\usepackage{doi}
\usepackage{enumitem}
\usepackage{xcolor}
\usepackage{graphicx}
\usepackage{hyperref}
\usepackage[nameinlink,capitalise]{cleveref}
\usepackage[margin=1in,marginparwidth=0.8in, marginparsep=0.1in]{geometry}
\usepackage{todonotes}
\hypersetup{
    colorlinks,
    linkcolor={black!50!blue},
    citecolor={blue!50!black},
    urlcolor={blue!80!black}
}

\newcommand{\Lx}{\left(}
\newcommand{\Rx}{\right)}
\newcommand{\LB}{\left[}
\newcommand{\RB}{\right]}

\newcommand{\R}{\mathbb{R}}
\newcommand{\N}{\mathbb{N}}

\newcommand{\Z}{\mathbb{Z}}
\newcommand{\F}{\mathbb{F}}

\newcommand{\C}{\mathbb{C}}

\newcommand{\res}{\!\!\mid}

\newtheorem{theorem}{Theorem}[section]
\newtheorem{lemma}[theorem]{Lemma}

\newtheorem{proposition}[theorem]{Proposition}
\newtheorem{definition}[theorem]{Definition}
\newtheorem{corollary}[theorem]{Corollary}

\newtheorem*{theorem*}{Theorem}
\newtheorem*{lemma*}{Lemma}
\newtheorem*{hypothesis*}{Hypothesis}
\newtheorem*{proposition*}{Proposition}
\newtheorem*{definition*}{Definition}
\newtheorem*{corollary*}{Corollary}
\newtheorem*{conjecture*}{Conjecture}
\newtheorem*{claim*}{Claim}

\theoremstyle{remark}
\newtheorem{remark}[theorem]{Remark}
\newtheorem*{remark*}{Remark}
\newtheorem{example}[theorem]{Example}

\numberwithin{equation}{section}

\DeclarePairedDelimiterX{\set}[1]{\{}{\}}{\setargs{#1}}
\NewDocumentCommand{\setargs}{>{\SplitArgument{1}{;}}m}
{\setargsaux#1}
\NewDocumentCommand{\setargsaux}{mm}
{\IfNoValueTF{#2}{#1} {#1\,\delimsize|\,\mathopen{}#2}}

\let\badphi\phi
\let\phi\varphi
\let\varphi\badphi

\usepackage{tikz}
\usepackage{tabularx,array}
\usetikzlibrary{intersections}
\usetikzlibrary{calc}
\tikzset{vtx/.style={circle, fill, inner sep=1.5pt}}

\DeclareMathOperator{\sh}{Sh}
\DeclareMathOperator{\mush}{\mu Sh}

\DeclareMathOperator{\perf}{Perf}
\DeclareMathOperator{\prop}{Prop}
\DeclareMathOperator{\der}{\mathbf{D}\!}

\DeclareMathOperator{\rep}{\mathbf{Rep}}
\DeclareMathOperator{\kk}{\mathbf{k}}
\DeclareMathOperator{\vect}{\mathbf{Vect}}
\DeclareMathOperator{\m}{m\!}

\DeclareMathOperator{\dgcat}{\mathbf{dg-Cat}}
\DeclareMathOperator{\musupp}{\mu supp}
\DeclareMathOperator{\cat}{\mathbf{Cat}}
\DeclareMathOperator{\icat}{\infty-\mathbf{Cat}}
\newcommand{\cyl}{\mathbb{S}}
\renewcommand{\res}[1]{\!\!\mid_{#1}}
\DeclareMathOperator{\sym}{Sym}
\newcommand{\eps}{\(\epsilon\)}
\DeclareMathOperator{\mods}{\mathbf{Mod}}
\DeclareMathOperator{\fun}{Fun}
\DeclareMathOperator{\ext}{Ext}
\DeclareMathOperator{\muhom}{\mu hom}

\DeclareMathOperator{\DS}{DS}
\DeclareMathOperator{\id}{Id}

\DeclareMathOperator{\codim}{codim}
\DeclareMathOperator{\link}{lk}
\DeclareMathOperator{\image}{im}
\DeclareMathOperator{\ob}{Ob}
\DeclareMathOperator{\dgcatmor}{\mathbf{dg-Cat}^{Mor}}
\DeclareMathOperator{\cc}{\mathcal{C}}
\DeclareMathOperator{\lgr}{LGr}
\DeclareMathOperator{\supp}{supp}
\newcommand{\sk}{\mathfrak{sk}}

\begin{document}
\title{Legendrian non-squeezing via microsheaves}
\author{Eric Kilgore}
\date{}
\maketitle

\begin{abstract}
We show that Legendrian pre-quantization lifts of many non-exact Lagrangian submanifolds in \( \C^n \) retain some quantitative rigidity from the symplectic base. In particular, they cannot be moved by Legendrian isotopy into an arbitrarily small pre-quantized cylinder. This is a high dimensional generalization of results of Dimitroglou Rizell--Sullivan in dimension 3. In this setting, we give a new proof of non-squeezing using normal rulings, and in high dimension, we obtain our results using a category of (micro)sheaves associated to a Legendrian submanifold of pre-quantizations.
\end{abstract}

\setcounter{tocdepth}{1}

\tableofcontents

\setcounter{tocdepth}{2}

\section{Introduction}\label{sec:intro}
Symplectic embeddings have long been a topic of considerable interest to symplectic geometers beginning with the celebrated work of Gromov \cite{Gromov85}. Since Chekanov's work on Lagrangian tori in \( \C^n \) (\cite{Chekanov-energy}) it has been known that (monotone)Lagrangian submanifolds also have an associated size, and there has since been considerable progress in obstructing Lagrangian isotopies of, say, a Clifford torus into too small a ball (\cite{CieliebakMohnke-Lagrangians}). Naively there is no contact analogue of the non-squeezing phenomenon, the unit ball in \( (\R^{2n-1},\xi_{std}) \) is contact isotopic to a ball of (e.g.) half the size, but Eliashberg-Kim-Polterovich discovered, \cite{EKP06}, that some remnant of symplectic embedding rigidity persists in the pre-quantization of a Liouville domain. The subject of the present article is a Legendrian analog of this story.

Let \( L \subset (\C^n,d\lambda) \) be a closed, embedded Lagrangian such that \( \LB\lambda\res{L}\RB \in H^1(L;\Z). \) Then \( L \) admits a Legendrian lift \( \Lambda \) to the pre-quantization space \( \widehat{\C^n} := (S^1 \times \C^n,\ker(dt + \lambda)), \) i.e. there exists a closed Legendrian \( \Lambda \subset \widehat{\C^n} \) whose projection into \( \C^n \) is a diffeomorphism onto \( L. \) The question is then as follows: Suppose \( L \) as above does not admit a Hamiltonian isotopy into a domain \( U \subset \C^n. \) Does a Legendrian lift \( \Lambda \) admit a Legendrian isotopy into the pre-quantization domain \( \widehat{U}. \) This was first addressed in dimension three by Dimitroglou Rizell--Sullivan in \cite{RizellSullivan-limits} via a local Thurston--Bennequin inequality, but remains untreated in high dimensions.

Our main theorem provides an answer for lifts of many high dimensional Lagrangians:
\begin{theorem}\label{thm:high-d-squeezing}
    Let \( A \in \N, \) \( L \subset \C^n \) be one of the following:
    \begin{itemize}
        \item The monotone twist torus obtained from a loop enclosing area \( A \) in \( \C, \) or the monotone Clifford torus \( T^2_{cl}(A,A) \subset \C^2. \)
        \item The Maslov-2 resolution of the Lagrangian Whitney sphere of monotonicity constant \( A. \)
        \item A non-monotone, integral Clifford torus \( T^n_{cl}(A_1,\ldots,A_n) \), with \( A_i > A_1 \), and \( A_i \in A_1\Z \) for all \( i > 1. \)
    \end{itemize}
    Then any Legendrian lift \( \Lambda \subset \widehat{\C^n} \) of \( L \) does not admit a Legendrian isotopy into the cylinder \( \widehat{Z}(A) := S^1 \times D(A) \times \C^{n-1}. \)
\end{theorem}
Monotone twist tori are exactly those which arise from iterated application of Chekanov's twisting construction (\cite{Chekanov-tori,Chekanov-Schlenk}), for more discussion see \cref{sec:twist-torus-fronts}.

In dimension three, this recovers some special cases of the Legendrian non-squeezing result of \cite{RizellSullivan-limits}; namely when the Legendrian considered is a lift of an embedded circle bounding a disk of integer area. This low dimensional case can be handled by elementary methods, which serve as an instructive warm-up for the techniques necessary in the general setting. 

The squeezing obstruction in this setting is provided by an inequality between the counts of two types of \emph{normal ruling} (in the sense of \cite{ChekanovPushkar05}) one can define for Legendrians in \( \widehat{T^*\R} \) (with front projection on a cylinder). One is defined simply as a pseudo-involution on the front, respecting certain conditions at crossings and cusps, while the other also records a preferred path between paired strands, in exchange for a slight relaxation of the crossing condition. When a Legendrian is sufficiently squeezable, every ruling of the first type induces one of the second, providing the aforementioned inequality, and thus providing a squeezing obstruction, once one verifies that the counts are Legendrian isotopy invariants. This approach is the content of \cref{sec:cylindrical-rulings}.

Beyond dimension three, \cref{thm:high-d-squeezing} is an application of a general squeezing obstruction for a Legendrian submanifold \( \Lambda \subset \widehat{\C^n} \) coming from the category of microsheaves supported on \( \Lambda: \) Such a Legendrian gives rise to an open cylindrical Lagrangian \( S\Lambda \) in the symplectization. In the pre-quantization setting, there are two natural completions of this cylinder to (singular) isotropic cones, \( S\Lambda^{\C} \) and \( S\Lambda^{\cyl}, \) obtained, respectively, by collapsing to negative end on to a point, and on to a copy of the \( S^1 \) factor of the pre-quantization. To such a cone, and a choice of field \( \kk \) the machinery of \cite{NadlerShende-weinstein} associates an \( \kk \) linear dg-category, which is an invariant of the Legendrian isotopy type (in \( \widehat{\C^n} \)) of \( \Lambda. \) We denote these by \( \sh^{\C}_{\Lambda} \) and \( \sh^{\cyl}_{\Lambda} \) respectively.

These dg-categories have distinguished full sub-categories, denoted \( \sh^{\bullet,1}_{\Lambda} \) (\( \bullet \in \set*{\C,\cyl} \)), whose objects are the microsheaves with microstalks of rank 1. In favorable circumstances, for example, when \( \Lambda \) is closed, smooth, and embedded, and working over a finite field \( \F_q, \) the homology category \( H\sh^{\bullet,1}_{\Lambda} \) has finitely many isomorphism classes of objects, and so one can associate a count of objects which is a Legendrian isotopy invariant. We denote this number by \( \#\sh^{\bullet,1}_{\Lambda}. \) The desired squeezing obstruction then comes from an inequality in these counts for Legendrians which are geometrically small:
\begin{theorem}\label{thm:pushing-out}
Suppose \( \Lambda \subset \widehat{\C}^n \) is a smooth, embedded, Legendrian contained in \( \widehat{Z}(1-\epsilon) \). Then
\[ \#\sh^{\C,1}_{\Lambda} \leq \#\sh^{\cyl,1}_{\Lambda}. \]
\end{theorem}
\begin{remark}
    There are many ways of trying to count objects in a dg-category which account for more or less structure (see, e.g., \cite{GorskyHaiden24}). For our present, geometric, purposes one may safely think of the above inequality in the most na\"ive sense i.e. a literal count of isomorphism classes. On the other hand we expect that the proof of \cref{thm:pushing-out} can be refined to give a more precise comparison of objects, which would yields the same inequality when the na\"ive count is replaced with an orbifold count, or the homotopy cardinality. 
\end{remark}
The proof of \cref{thm:pushing-out} proceeds via a local-to-global lifting construction, which builds an object of \( \sh^{\cyl,1}_{\Lambda} \) out of an object of \( \sh^{\C,1}_{\Lambda}. \) The geometric assumption \( \Lambda \subset \widehat{Z}(1-\epsilon) \) is translated into a sort of ``locality'' (in the sense of the front projection) for objects of \( \sh^{\C,1}_{\Lambda}, \) which in turn is used to guarantee that a collection of local lifts can be (canonically) glued into a global object of \( \sh^{\cyl,1}_{\Lambda}. \) This occupies the entirety of \cref{sec:essential-rulings}.


\subsection{Outline of the paper}\label{sec:outline}
In \cref{sec:geom-prelim} we review some basics of contact geometry, and collect a few facts about fronts of generic Legendrians. \cref{sec:cylindrical-rulings} is devoted to a (brief) proof of three dimensional Legendrian non-squeezing using our approach, intended to serve as a summary of the strategy unburdened by the additional technicality of the microsheaf formalism. In the process we introduce some new variants of rulings for Legendrian knots beyond the \( \R \) valued jet.

\cref{sec:preliminaries} is devoted to collecting algebraic and categorical preliminaries for our discussion of microsheaves, which are introduced in \cref{sec:microsheaves}. In \cref{sec:sheaf-invariants} and \cref{sec:combinatorics} we introduce the relevant forms of microsheaves that we consider in order to approach Legendrian non-squeezing in our setting, and collect some local computations useful for getting a handle on objects. In \cref{sec:essential-rulings} we prove \cref{thm:pushing-out} as sketched above. Finally, in \cref{sec:non-squeezing} we use the results of the previous section to prove our main result, \cref{thm:high-d-squeezing}.

\bf Acknowledgments \rm The author would like to thank Yasha Eliashberg for introducing him to the work of Chekanov--Pushkar, and many helpful discussions, and David Nadler for suggesting the use of microsheaves to go beyond dimension 3. The author also thanks Maya Sankar for showing him how to use Tikz to generate most of the figures that appear in this paper.

\section{Geometric preliminaries}\label{sec:geom-prelim}
Throughout, we identify the circle with the quotient of \( \R \) by \( \Z \): \( S^1 \!:= \R/\Z. \) In particular, we treat it as having length 1.

\subsection{Recollections on pre-quantizations}\label{sec:pre-quantizations}
Let \( (W,d\lambda) \) be a Liouville manifold. Its \emph{pre-quantization} is the contact manifold 
\[ \widehat{W} := (S^1 \times W, \ker(dt + \lambda)). \]
It carries a standard contact form \( \alpha = dt + \lambda \) whose Reeb flow is rotation in the \( S^1 \) factor, and is equipped with a canonical projection \( \varpi:\widehat{W} \rightarrow W \) which forgets the \( S^1 \) factor.

If \( \iota:L \rightarrow W \) is a smooth Lagrangian embedding, such that \( \iota^*\lambda \in H^1(L;\Z), \) then for any choice of \( l_0 \in L \) the map 
\[ \tilde\iota:L \rightarrow \widehat{W} \qquad \tilde\iota(l) = \Lx \int_{l_0}^{l} \lambda, \iota(l) \Rx \]
is a smooth Legendrian embedding such that \( \varpi \circ \tilde\iota = \iota. \) In this case, we say that \( \Lambda := \tilde\iota(L) \) is a Legendrian \emph{lift} of \( L. \) Note that lifts of such \( L \) are unique up to Legendrian isotopy: all are related by some (global) rotation of the \( S^1 \) factor, which is a Reeb flow as described above. Moreover, if \( L' \) is Hamiltonian isotopic to \( L, \) then any lift \( \Lambda' \) of \( L' \) is Legendrian isotopic to a (thus any) lift \( \Lambda \) of \( L. \)

In the case \( W = T^*M, \) a cotangent bundle equipped with its tautological 1-form, there is a canonical contactomorphism 
\[ \widehat{W} \cong J^1(M;S^1) \]
between the pre-quantization of \( W, \) and the circle-valued 1-jet space of \( M. \)

We will habitually identify \( \widehat{\C}^n \) with \( \widehat{T^*\R^n} \cong J^1(\R^n;S^1). \) This is justified by an application of Gray's theorem: Any Legendrian isotopy of a closed Legendrian can be realized by a compactly supported contactomorphism. The two contact forms induced by Liouville homotopic Liouville forms are isotopic through contact forms, and thus yield contactomorphic contact manifolds on any compact subset.

\subsection{Legendrian fibrations}\label{sec:legendrian-fibrations}
Let \( (Y,\xi) \) be a contact manifold. A Legendrian fibration is a smooth submersion \( p:Y \rightarrow F \) whose fibers are Legendrian submanifolds. We will refer to the base \( F \) as the \emph{front space} and the map \( p \) as the \emph{front projection}. When \( Y \) is a 1-jet, \( Y = J^1(X;Z) \) for some \( Z \in \set*{\R,\S^1}, \) and \( p:Y \rightarrow X \times Z \) is the standard Legendrian fibration by cotangent fibers, there is a further projection \( \Pi:Y \rightarrow \mathfrak{B} := X. \) We refer to this as the \emph{bifurcation projection}, and to the base \( \mathfrak{B} \) as the \emph{bifurcation space}. The intermediate projection from the front to bifurcation space is denoted \( \pi: F \rightarrow \mathfrak{B} \). While fixing further notation (as in the rest of this section) we will be careful to distinguish between \( \Lambda \) and its image under the front or bifurcation projection, but later on we may conflate these, and denote all by \( \Lambda \) when it is clear where we wish to view it from context.

\subsection{Stratification of front singularities}\label{sec:singularity-stratification}
A subspace \( A \subset X \) is called \textit{smoothly stratified} if it is equipped with a filtration \( A_0 \subset A_1 \subset \ldots \subset A_N = A \) such that \( A_i \setminus A_{i-1} \) is a smooth manifold for every \( i. \) 

Let \( Y = J^1(X;Z) \) be a standard Legendrian fibered jet space, \( p,F,\Pi,\mathfrak{B} \) as in \cref{sec:legendrian-fibrations}. The \textit{frontal singular locus} \( \tilde{\mathbb{X}}_{\Lambda} \) of \( p(\Lambda) \) is defined to be the set of points \( x \in F \) where either \( Dp \) drops rank, or over which \( p \) fails to be injective. By standard jet transversality, and the abundance of Legendrian perturbations, \( \tilde{\mathbb{X}}_{\Lambda} \) is a smoothly stratified, codimension at least 2 subspace of \( F, \) for generic \( \Lambda. \) Observe that \( \tilde{\mathbb{X}}_{\Lambda} \) is \textit{automatically} transverse (in the sense of smoothly stratified spaces) to the fibers of \( \pi, \) since \( \Pi(\Lambda) \) is on the complement of \( \tilde{\mathbb{X}}_{\Lambda} \). 

The \textit{bifurcating locus}, \( \mathbb{X}_{\Lambda}, \) of \( \Lambda \) is defined to be \( \pi(\tilde{\mathbb{X}}). \) Keeping in mind the automatic transversality of \( \tilde{\mathbb{X}}_{\Lambda}, \) we conclude that, after further generic perturbation, \( \mathbb{X}_{\Lambda} \) is also smooth, stratified, and now codimension 1. Note that this stratification is not quite the same as that on \( \tilde{\mathbb{X}}_{\Lambda}, \) but is induced by the projection of a refinement of this stratification, where the additional strata are precisely the pre-images of the subset where \( \pi\res{\tilde{\mathbb{X}}_{\Lambda}} \) fails to be injective, which are of positive codimension in \( \mathbb{X}_{\Lambda} \). Compatibility of these stratifications induces a partition into the \textit{pointlike}, and \textit{XX} loci, the latter defined to be the sub-stratified sub-space of \( \mathbb{X}_{\Lambda} \) over which \( \pi \) is not injective, with the former comprising the rest.

The following lemma will be useful later:
\begin{lemma}\label{lem:stratum-local-constancy}
    Fix a smooth stratum \( \mathfrak{x} \subset \mathbb{X}_{\Lambda}. \) Then over a sufficiently small (without loss of generality tubular) neighborhood \( U \ni x \in \mathfrak{x} \) the front \( \Lambda \) is isotopic to the trace of a \( \codim(\mathfrak{x}) \)-parametric family of Legendrian isotopies of the coisotropic reduction of its intersection with (the pre-image of) a local transverse slice (of complementary dimension) to \( \mathfrak{x}. \)
\end{lemma}
\begin{proof}
    This is a consequence of transversality of our slice, and smoothness of the underlying Legendrian.
\end{proof}

\subsection{Polarizations and twisted Maslov potentials}\label{sec:polarizations}
We discuss here the geometric aspects of a polarization for a contact manifold, and corresponding notion of (twisted) Maslov potentials. We use the word \emph{twisted} to emphasize that we allow ourselves work with respect to a non-standard polarization. We will revisit these objects in greater detail (and generality) in \cref{sec:sheaf-polarizations}.

A (co-oriented) contact manifold \( (Y^{2n+1},\xi) \) carries a canonical conformal symplectic structure on the contact distribution \( \xi \subset TY. \) In particular, \( \xi \) can be regarded as a rank \( 2n \) real symplectic vector bundle over \( Y. \) To such a vector bundle there are two associated fiber bundles, \( \lgr(Y), \lgr^+(Y), \) with fibers \( \lgr_n \) and \( \lgr^+_n, \) respectively the ordinary and oriented Lagrangian Grassmannians of the symplectic vector space of real dimension \( 2n \). The latter double covers the former by construction, and so induces a double covering on the level of the Lagrangian--Grassmann bundles. A \emph{polarization} (resp. \emph{orientable polarization}) of \( Y \) consists of a section of \( \lgr(Y) \) (\( \lgr^+(Y) \)), i.e. a smoothly varying choice of (oriented) Lagrangian subspace of \( \xi \) over each point of \( Y. \) When \( (Y,\xi) \) admits an (orientable) polarization it is said to be \emph{(orientably) polarizable}.

\begin{remark}
    A (orientable) polarization has the following bundle theoretic interpretation: A symplectic vector bundle \( E \rightarrow B \) of rank \( 2n \) is, equivalent to the data of a \( U(n) \) principal bundle. Then a (orientable) polarization \( \tau:B \rightarrow \lgr^{(+)}(E) \) induces a trivialization of the associated Lagrangian--Grassmann bundle
    \[ \lgr(E) \cong B \times U(n)/O(n) \quad (\mathrm{resp. } \lgr^+(E) \cong B \times U(n)/SO(n)) \]
    by the action on the distinguished Lagrangian subspaces determined by \( \tau. \) Let \( \sigma:\lgr(E) \rightarrow U(n)/O(n) \) denote the projection of the trivializing map onto the second factor. We then obtain a ``global'' Maslov index on \( E \) via the composition
    \[ \mu^{\tau}_E := \left(\sigma^* \circ \iota\right)(\mu), \]
    where \( \mu \in H^1(U/O) \) is the universal Maslov class, and \( \iota:U(n)/O(n) \rightarrow U/O \) is induced by stabilization. In particular, when \( B \) is a contact (or symplectic) manifold, and \( E \) is its contact distribution (tangent bundle), a (oriented) polarization induces a class on the total space of the associated (oriented) Lagrangian--Grassmann bundle whose fiberwise restriction is exactly the Maslov class of the fiber. 
\end{remark}

Now suppose \( (Y,\xi) \) is polarizable, and fix a polarization \( \tau:Y\rightarrow \lgr(Y). \) A Legendrian submanifold \( \Lambda \subset Y \) comes with a canonical \emph{Gauss map}, \( G_{\Lambda}:\Lambda \rightarrow \lgr(Y) \) given by its tangent space. If \( \Lambda \) is orientable, then a choice of orientation yields a lift of \( G_{\Lambda} \) to \( \lgr^+(Y). \) In this case, fixing a basepoint \( x_0 \in \Lambda \) (assume, without loss of generality, that \( G_{\Lambda}(x_0) \pitchfork \tau(x_0) \)) we obtain a map \( \pi_1(\Lambda,x_0) \rightarrow \Z, \) defined by evaluating \( \mu_{\xi} \) on the image of \( (G_{\Lambda})_{*}\pi(\Lambda,x_0) \) under the Hurewicz map. Since \( G_{\Lambda} \) admits a lift to the double cover, the image of this map lands in \( 2\Z, \) and so can be characterized completely by an even integer \( n^{\tau}_{\Lambda}, \) the \emph{Maslov number} of \( \Lambda, \) which generates this image. 

Now if \( x \in \Lambda \) is any other point, lifting \( \mu_{\xi} \) in the long exact sequence of the pair \( (\lgr(\Lambda),G_{\Lambda}(x_0) \cup G_{\Lambda}(x)) \) over \( \Z/n^{\tau}_{\Lambda}\Z, \) and evaluating on any class in \( H_1(\Lambda,x_0\cup x;\Z) \) induced by a path \( x_0 \mapsto x \) yields a well defined element of \( \Z/n^{\tau}_{\Lambda}\Z. \) We define the \emph{Maslov potential twisted by \( \tau \) based at \( x_0 \)} on \( \Lambda \) to be the \( \Z/n^{\tau}_{\Lambda}\Z \) valued function on \( \Lambda \) obtained via this procedure, and denote it by \( m_{(\Lambda,x_0)}^{\tau} \). Note that, by Arnol'd's geometric characterization of the universal Maslov class (\cite{Arnold-maslov}) \( m_{(\Lambda,x_0)}^{\tau} \) is locally constant over the region where \( G_{\Lambda} \pitchfork \tau \) (where the images are regarded as Lagrangian subspaces). We leave as an exercise to verify that when \( Y = J^1(\R;\R), \) \( \tau \) the polarization induced by cotangent fibers, and \( \Lambda \subset Y \) some front-generic Legendrian knot, \( n^{\tau}_{\Lambda} \) recovers (up to sign, twice the) the classical rotation number of \( \Lambda, \) and our notion of twisted Maslov potential reduces to the usual notion as defined in, say, \cite{ChekanovPushkar05}.

Note that, in the end, all relevant objects were pulled back from the stable Lagrangian--Grassmannian, so in the above description we could instead have asked only for a \emph{stable polarization}, i.e. a section of the \( U/O \) bundle obtained by the limit of stable trivializations of \( \xi. \) Similarly, a Legendrian carries a \emph{stable Gauss map} obtained by stabilizing \( G_{\Lambda}, \) and the remainder of the story is as before. Note that, in practice, one can always truncate at a finite stage in the limit defining \( U/O, \) so the above ``local constancy on the transverse locus'' can still be made sense of directly. Indeed, we will use the fact that any stable polarization is already defined after stabilizing by a finite rank trivial bundle in our definition of microsheaves later on.

\subsection{Dimension 3}\label{dim-3-prelim}
In dimension 3 the bifurcation stratification can always be made very simple: generically \( \mathbb{X}_{\Lambda} \subset \mathfrak{B} \) is a finite collection of points, each corresponding to a transverse crossing of two smooth strands, or a fold singularity. Moreover a generic Legendrian isotopy features only the following three pointlike codimension 2 front singularities, which are captured by Legendrian Reidemeister moves:
\begin{itemize}
    \item \textbf{R1}, the swallowtail singularity
    \item \textbf{R2}, the intersection (in the front) of a fold with a smooth point
    \item \textbf{R3}, a triple intersection of smooth strands.
\end{itemize}
The bifurcation singularities consist of these, alongside the coincidence (i.e. occurring in the same fiber of \( \pi \)) of two generic front singularities. These describe the XX locus, and consist of:
\begin{itemize}
    \item \( \mathfrak{XX} \), the coincidence of two simple crossings.
    \item \( \mathfrak{XC} \), the coincidence of a simple crossing and a fold.
    \item \( \mathfrak{CC} \), the coincidence of two folds.
\end{itemize}
Since these are all the singularities which occur in a generic 1-parameter family, in order to check Legendrian isotopy invariance of an object defined in terms of a front, it suffices to verify that it persists under ambient fibered isotopy, and these six bifurcations. This is the strategy employed by Chekanov--Pushkar to show invariance of the count of normal rulings (i.e. positive, Maslov, pseudo-involutions) \cite{ChekanovPushkar05}. This will also be our strategy in \cref{sec:cylindrical-rulings}. 

\begin{figure}[h]
\begin{center}
\begin{tabularx}{0.95\textwidth}{|>{\centering}X|>{\centering}X|}
\hline
\vspace{1\baselineskip}
    \begin{tikzpicture}
        \draw (0,.25) -- (2,.25);
        \draw[->] (2.5,.25) -- (4,.25);
        \draw (4.5,0) to[out=0,in=-135] (5.5,.25) to[out=45,in=180] (6,.5);
        \draw (5,.5) to[out=0,in=180] (5.5,.75) to[out=0,in=180] (6,.5);
        \draw (5,.5) to[out=0,in=-45] (5.5,.25) to[out=-45,in=180] (6.5,0);
    \end{tikzpicture}
    \par \textbf{R1} &
    \vspace{0.1\baselineskip}\begin{tikzpicture}
        \draw (0,0.1) to[out=0,in=225] (1,.5) to[out=45,in=180] (2,0.9);
        \draw (0,0.9) to[out=0,in=135] (1,.5) to[out=-45,in=180] (2,0.1);

        \draw (0,-.9) to[out=0,in=225] (1,-.5) to[out=45,in=180] (2,-.1);
        \draw (0,-.1) to[out=0,in=135] (1,-.5) to[out=-45,in=180] (2,-.9);

        \draw[dashed] (1,1) -- (1,-1);
    \end{tikzpicture}
    \par \( \mathfrak{XX} \)
    \tabularnewline
\hline
\vspace{1\baselineskip}
        \begin{tikzpicture}
            \draw (0,.5) to[out=-45,in=180] (1,0);
            \draw (0,-.5) to[out=45,in=180] (1,0);
            \draw (0.5,.75) -- (1.25,0) -- (2,-.75);

            \draw[->] (2,0) to (4,0);

            \draw (4.5,.5) to[out=-45,in=180] (5.5,0);
            \draw (4.5,-.5) to[out=45,in=180] (5.5,0);
            \draw (4.5,.75) -- (5.25,0) -- (6,-.75);
        \end{tikzpicture}
        \par \textbf{R2} &
        \vspace{.1\baselineskip}\begin{tikzpicture}
            \draw (0,0.2) to[out=0,in=225] (1,.6) to[out=45,in=180] (2,1);
            \draw (0,1) to[out=0,in=135] (1,.6) to[out=-45,in=180] (2,0.2);

            \draw (2,0) to[out=-135,in=0] (1,-.5);
            \draw (2,-1) to[out=135,in=0] (1,-.5);
            
            \draw[dashed] (1,1) -- (1,-1);
        \end{tikzpicture}
        \par \( \mathfrak{XC} \)
        \tabularnewline
\hline
\vspace{1\baselineskip}
    \begin{tikzpicture}
        \draw (0,0) to[out=0,in=225] (1,.33) to[out=45,in=180] (2,1);
        \draw (0,1) to[out=0,in=135] (1,.33) to[out=-45,in=180] (2,0);
        \draw (0,.67) -- (2,.67);

        \draw[->] (2.5,.5) -- (4,.5);

        \draw (4.5,0) to[out=0,in=225] (5.5,.67) to[out=45,in=180] (6.5,1);
        \draw (4.5,1) to[out=0,in=135] (5.5,.67) to[out=-45,in=180] (6.5,0);
        \draw (4.5,.33) -- (6.5,.33);
    \end{tikzpicture}
    \par \textbf{R3} &
    \vspace{.1\baselineskip}\begin{tikzpicture}
        \draw (0,1.5) to[out=-45,in=180] (1,1);
        \draw (0,.5) to[out=45,in=180] (1,1);

        \draw (2,.5) to[out=-135,in=0] (1,0);
        \draw (2,-.5) to[out=135,in=0] (1,0);

        \draw[dashed] (1,1.5) -- (1,-.5);
    \end{tikzpicture}
    \par \( \mathfrak{CC} \)
    \tabularnewline
    \hline
\end{tabularx}
\caption{Some fronts for the generic three dimensional bifurcations.}
\end{center}
\end{figure}
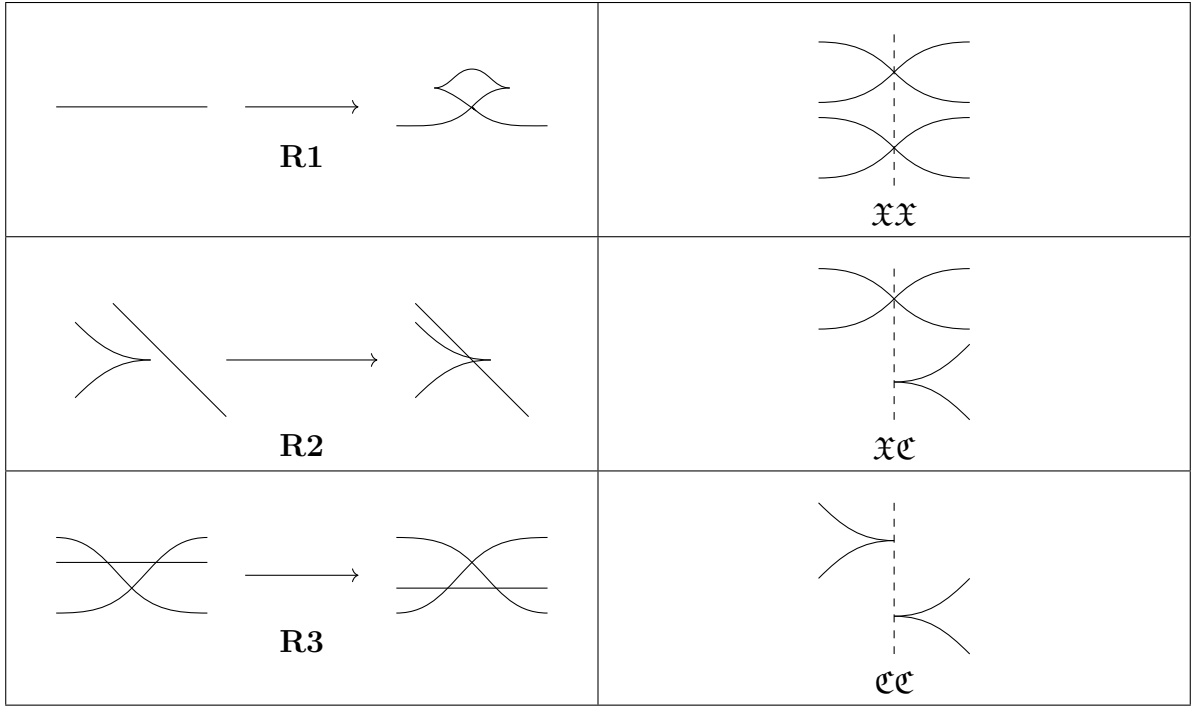

\section{Ruling fronts in \( S^1 \times \R \)}\label{sec:cylindrical-rulings}
In this section we prove a Legendrian non-squeezing result in dimension three using more elementary, combinatorial invariants defined in terms of the front projection of a Legendrian link. The geometric idea of the proof is, ultimately, the same as in high dimensions, so we include this approach as a warm up which is uncomplicated by the more technical language necessary to work with microsheaves.

\subsection{Fronts in \(J^1(\R;\R)\)}
We begin with a review of rulings for Legendrians in \( J^1(\R;\R), \) i.e. upwardly co-oriented fronts in \( \R \times \R. \) Let \( \Lambda \subset J^1(\R;\R) \) be a front-generic Legendrian. Then, as described in \cref{dim-3-prelim}, \( \tilde{\mathbb{X}}_{\Lambda} \) is discrete, finite, and partitioned into \( \Lambda_{\mathfrak{X}} \sqcup \Lambda_{\mathfrak{C}}, \) the crossing and fold loci respectively. Denote the complement of \( \tilde{\mathbb{X}}_{\Lambda} \) in \( p(\Lambda) \) (i.e. the smooth points of the front) by \( \Lambda_{\mathfrak{S}}. \)

Now over each \( x \in \mathfrak{x} \) the fiber of the bifurcation projection \( \pi^{-1}(x) \) is a disjoint union of points.
\begin{definition}[\cite{ChekanovPushkar05}]\label{def:pseudoinvolution}
    A weak pseudo-involution on the front \( p(\Lambda) \) is a continuous map \( \iota:\Lambda_{\mathfrak{S}} \cup \Lambda_{\mathfrak{C}} \rightarrow p(\Lambda) \) such that
    \begin{enumerate}
        \item[(i)] \( \pi = \pi \circ \iota. \)
        \item[(ii)] \( \iota \circ \iota(x) = x \) whenever \( \iota(x) \in \Lambda_{\mathfrak{S}} \cup \Lambda_{\mathfrak{C}}. \) 
        \item[(iii)] \( \iota(x) = x \iff x \in \Lambda_{\mathfrak{C}}. \)
    \end{enumerate}
    A weak pseudo-involution is called a pseudo-involution if it additionally satisfies
    \begin{enumerate}
        \item[(iv)] For every \( x \in \Lambda_{\mathfrak{X}}, \) there exists a neighborhood \( U \) such that \( \iota(p(\Lambda) \cap U) \cap \left(\Lambda_{\mathfrak{S}}\cap U\right) = \emptyset. \)
    \end{enumerate}
\end{definition}
This last condition imposes that the two smooth strands near a point of \( \Lambda_{\mathfrak{X}} \) may not be exchanged by \( \iota. \)

Notice that \( \mathbb{X}_{\Lambda} \) partitions \( \mathfrak{B} \cong \R \) into a collection disjoint open intervals: \( \mathfrak{B} \setminus \mathbb{X}_{\Lambda} = \sqcup_{i = 1}^{N} I_{i}, \) without loss of generality ordered such that \( \overline{I}_i \cap \overline{I}_j = \emptyset \) unless \( i = j\pm 1. \) So for any collection \( \set*{x_i}_{i=1}^N \) with \( x_i \in I_i \), by continuity a pseudo-involution on \( p(\Lambda) \) is determined by its action on the fibers \( \pi^{-1}(x_i). \) 

Suppose \( \mathfrak{x} \in \Lambda_{\mathfrak{X}} \) lies over the intersection \( \overline{I}_i \cap \overline{I}_{i+1}. \) Let \( \sigma_i^{\pm} \) denote the two strands of \( \Lambda_{\mathfrak{S}} \cap \pi^{-1}(I_i) \) whose closure contains \( \mathfrak{x} \) (ordered such that the short path along the co-orientation direction near \( \mathfrak{x} \) runs from \( \sigma_{i}^{-} \) to \( \sigma_{i}^+ \)), and let \( \tau_{i}^{\pm} \) denote the component in \( \Lambda_{\mathfrak{S}} \) of their respective images under \( \iota. \) Similarly let \( \sigma_{i+1}^{\pm} \) denote the strands over \( I_{i+1} \) with common boundary \( \mathfrak{x}, \) ordered as before and \( \tau_{i+1}^{\pm} \) their images. By continuity, either \( \tau_{i}^{\pm} = \tau_{i+1}^{\pm} \) or \( \tau_{i}^{\pm} = \tau_{i}^{\mp}. \) In the former case the crossing \( \mathfrak{x} \) is called \emph{switching}, and in the latter it is called \emph{non-switching}.

Fix some \( x_j \in I_j, j \in \set*{i,i+1}. \) Then
\[ S_{j}^{\pm} := \left(\sigma_{j}^{\pm} \cup \tau_{j}^{\pm}\right) \cap \pi^{-1}(x_j) \]
define a pair of embeddings of \( S^0 \) into \( \pi^{-1}(x_j) \) for each \( j. \) Suppose \( \mathfrak{x} \) is switching. Then we call \( \mathfrak{x} \) \emph{positive} if the \( \Z/2 \) linking numbers \( \link(S_{j}^{+},S_{j}^{-}) \equiv 0 \) for both \( j \in \set*{i,i+1}. \)

Let \( \mu \) be a \( \Z/2 \) valued Maslov potential\footnote{The notion of Maslov potential is unambiguous in this setting, since the jet is contractible.} on \( p(\Lambda). \) A switching crossing \( \mathfrak{x} \) is called \emph{Maslov} if \( \mu(\sigma_j^{+}) = \mu(\sigma_j^{-}). \) Note that if \( \mathfrak{x} \) is Maslov for a particular Maslov potential, it is so for every such choice.

\begin{definition}
    A pseudo-involution on \( p(\Lambda) \) is called a \emph{normal ruling} of \( \Lambda \) if every switching crossing is both positive and Maslov.
\end{definition}

The main technical result of \cite{ChekanovPushkar05} shows that the number of normal rulings of a front generic Legendrian link \( \Lambda \subset J^1(\R;\R) \) is a Legendrian isotopy invariant. The proof proceeds via a case-by-case analysis of the codimension 2 bifurcation singularities that appear in generic 1-parametric Legendrian isotopies, showing that, given a normal ruling of the front before the bifurcation, there exists a unique normal ruling of the front after the bifurcation which preserves the difference between the number of switching crossings, and folds:
\[ \chi(\iota) := \#\!\Lambda_{\mathfrak{X}}^{sw} - \frac{1}{2}\#\!\Lambda_{\mathfrak{C}}. \]

\subsection{Fronts in the cylinder}
Now suppose \( \Lambda \subset J^1(\R;S^1), \) equipped with the fiber polarization. As before, we assume that \( p(\Lambda) \) is front-generic, and so has only simple crossings and folds. The definition of (positive, Maslov) pseudo-involution, and thus normal ruling, above may be implemented verbatim. We call these objects \emph{normal disk rulings}, and the set of all such \( \mathrm{R}^{\C}(\Lambda). \) The count of these is invariant by repeating the arguments of \cite{ChekanovPushkar05}, and denoted by \( \#\!\mathrm{R}^{\C}(\Lambda). \)

\begin{remark}
    As will become clear in \cref{sec:sheaf-invariants} the ``correct'' polarization with respect to which normal disk rulings ought to be defined is not the fiber polarization, but rather one which differs from this by a single orientable twist. Regarding \( J^1(\R;S^1) \) as the boundary of the trivial open book on \( T^*\R \) minus it's binding, this polarization extends over the filling, while the fiber polarization does not. Nevertheless, as long as we work with a representative of this polarization which is invariant in the cotangent coordinate, and with respect to which our Legendrian knot is generic, and we consider only \( \Z/2 \) valued Maslov potentials, a switching crossing is Maslov for the fiber polarization if and only if it is so for the twisted one by orientability. As such we will continue to work in the fiber polarization, but only with \( \Z/2 \) valued potentials.
\end{remark}

There is another reasonable definition of ruling for Legendrians in the \( S^1 \)-valued jet: Let \( \iota \) be a weak pseudo-involution on \( p(\Lambda). \) 
\begin{definition}\label{def:circular-ruling}
    An admissible marking \( \Gamma \) for \( \iota \) consists of a homotopy class of path
    \[ [\gamma_x] \in \pi_1(\pi^{-1}(\pi(x)),\set*{x} \cup \set*{\iota(x)}) \]
    for each \( x \in \Lambda_{\mathfrak{S}}, \) such that
    \begin{itemize}
        \item If \( \iota(x) \in \Lambda_{\mathfrak{S}} \) then \( [\gamma_{\iota(x)}] = [\gamma_{x}]^{-1}. \)
        \item If \( x,y \in \sigma, \) a component of \( \Lambda_{\mathfrak{S}}, \) then \[ [\gamma_x] = [\gamma_y] \in \pi_1(\pi^{-1}(\pi(\sigma)),\sigma \cup \iota(\sigma)). \]
        \item If \( x \in \Lambda_{\mathfrak{X}} \) is not in the image of \( \iota, \) and \( \sigma_{L}^{\pm},\sigma_{R}^{\pm} \) are the smooth strands containing \( x \) in their closure on the left and right of \( x, \) then neither \( [\gamma_L] \) nor \( [\gamma_R], \) the left or right homotopy classes of path, are trivial in
        \[ \pi_1\left(\pi^{-1}\left( \pi(U_x) \right), U_x \cap p(\Lambda) \right), \]
        for a small neighborhood \( x \in U_x \subset F. \) 
        \item If \( x \in \Lambda_{\mathfrak{C}}, \) \( \sigma^{\pm} \) the smooth strands containing \( x \) in their closure, then for \( y \in \sigma^{\pm} \) \( [\gamma_y] \) is trivial in 
        \[ \pi_1\left(\pi^{-1}\left( \pi(U_x) \right), U_x \cap p(\Lambda) \right), \]
        for small neighborhoods \( U_x \) of \( x. \)
    \end{itemize}
\end{definition}
A \emph{circular ruling} of \( \Lambda \) is a weak pseudo-involution \( \iota \) on \( p(\Lambda), \) together with an admissible marking \( \Gamma \). 

If a circular ruling \( (\iota,\Gamma) \) satisfies condition (iv) of \cref{def:pseudoinvolution} at \( x \in \Lambda_{\mathfrak{X}}, \) then we can still categorize \( x \) as either switching or non-switching as before. In fact, even for those \( x \in \Lambda_{\mathfrak{X}} \) which are not in the image of \( \iota \) we can still recover the notions of switching and non-switching using the marking data \( \Gamma: \) A crossing in the front of a Legendrian link admits a unique (up to isotopy) horizontal smoothing. A crossing point is called switching if the marking data of \( (\iota,\Gamma) \) away from the crossing extends over the horizontal resolution of \( \mathfrak{x}, \) and non-switching otherwise.

As before, a ruling of \( p(\Lambda) \) can only be guaranteed to persist under Legendrian isotopy under further assumptions on its behavior at switching points, i.e. if it satisfies some analog of the positivity condition described above. The correct notion here is a slight relaxation of positivity in the sense above, and can be thought of as ``positivity on the universal cover:''
Let \( x \in \Lambda_{\mathfrak{X}} \) be a switching crossing of a circular ruling \( (\iota,\Gamma), \) and suppose \( x \in \image(\iota). \) Let \( \sigma_{L}^{\pm},\sigma_{R}^{\pm} \) denote the left and right pairs of smooth strands of \( p(\Lambda) \) limiting to \( x, \) and \( [\gamma_{\sigma_{L}^{\pm}}], [\gamma_{\sigma_{R}^{\pm}}] \) their markings. Fix a lift \( \tilde{U}_x \) of a neighborhood \( U_x \) of \( x \) to the universal cover of the front \( \tilde{F}. \) Denote by \( [\tilde{\gamma}_{\sigma_{\bullet}^{\pm}}] \) (\( \bullet \in \set*{L,R} \)) the lift of a marking path to the universal cover which is based in \( \tilde{U}_x. \) 
\begin{definition}
    The crossing \( x \) is called positive if
    \[ \link([\tilde{\gamma}_{\sigma_{L}^{+}}],[\tilde{\gamma}_{\sigma_{L}^{-}}]) \equiv \link([\tilde{\gamma}_{\sigma_{R}^{+}}],[\tilde{\gamma}_{\sigma_{R}^{-}}]) \equiv 0 \mod 2, \]
    where this linking number is defined on the universal cover over some nearby smooth points \( x_L,x_R \in \pi(U_x)\setminus \pi(x). \)
\end{definition}
By convention, if \( x \not\in \image(\iota), \) we still call \( x \) positive.

With this in hand we define a \emph{normal circular ruling} as a pair \( (\iota,\Gamma) \) of weak pseudo-involution and admissible marking which is positive and Maslov at switching crossings, and denote the set of all such by \( \mathrm{R}^{\cyl}(\Lambda). \) Note that such a ruling has no switching points in the complement of the image of \( \iota, \) since these would necessarily be non-Maslov.
\begin{proposition}\label{prop:circular-ruling-count}
    Every normal circular ruling of a Legendrian link \( \Lambda \subset J^1(\R;S^1) \) admits a unique (characteristic) continuation over isotopy bifurcations. In particular, the number of normal circular rulings, \( \#\!\mathrm{R}^{\cyl}(\Lambda) \) is a Legendrian isotopy invariant.
\end{proposition}

\begin{proof}
    The proof is by a case-by-case analysis as in \cite{ChekanovPushkar05}. Let \( (\iota_0,\Gamma_0) \) be a normal circular ruling of \( \Lambda. \) The only bifurcation for which continuation does not follow directly from the analysis of \cite{ChekanovPushkar05} is the \textbf{R3} move, in the presence of crossings in \( \Lambda_{\mathfrak{X}} \setminus \iota(p(\Lambda)). \) By normality, these are never switching, and in order to occur, exactly two of the strands involved in the \textbf{R3} move must be paired by \( \iota_0 \) on both the left and right sides of the input diagram (i.e. the ruling before undergoing bifurcation). It suffices to consider the following geometric picture:
    \begin{figure}[ht]
    \centering
        \begin{tikzpicture}
            \draw[name path=curve1, red]
            (0,0) to [bend right=35] (3,1) to [bend left=30] (5,2);
            \draw[name path=curve2, blue]
            (0,2) to [bend right=-30] (2,1) to [bend left=-35] (5,0);
            \draw[name path=curve3, black]
                (0,1.3) to[bend right=0] (5,1.3);
            \path[name intersections={of={curve1 and curve2}, by={P12}}];
            \path[name intersections={of={curve1 and curve3}, by={P13}}];
            \path[name intersections={of={curve2 and curve3}, by={P23}}];

            \draw[->] (2.5,0.3) -- node[right] {$\pi$} (2.5,-1.2);

            \path let \p1 = (P12) in coordinate[vtx] (X12) at (\x1,-1.5);
            \path let \p1 = (P13) in coordinate[vtx] (X13) at (\x1,-1.5);
            \path let \p1 = (P23) in coordinate[vtx] (X23) at (\x1,-1.5);

            \draw[name path=baseleft, black] (0,-1.5) -- node[below] {$S_1$} (X23) -- node[below] {$S_2$} (X12) -- node[below] {$S_3$} (X13) -- node[below] {$S_4$} (5,-1.5);
            \node[left] at (0,-1.5) {$\mathfrak{B}$};
            
            \draw[->] (5.5,0.5) to node[above] {\textbf{R3}} (7.5,0.5);

            \draw[name path=curve4, red] (8,0) to [bend right=30] (10,1) to [bend left=35] (13,2);
            \draw[name path=curve5, blue] (8,2) to [bend right=-35] (11,1) to [bend left=-30] (13,0);
            \draw[name path=curve6, black] (8,0.7) to [bend right=0] (13,0.7);

            \path[name intersections={of={curve4 and curve5}, by={P45}}];
            \path[name intersections={of={curve4 and curve6}, by={P46}}];
            \path[name intersections={of={curve5 and curve6}, by={P56}}];

            \draw[->] (10.5,0.3) -- node[right] {$\pi$} (10.5,-1.2);

            \path let \p1 = (P45) in coordinate[vtx] (X45) at (\x1,-1.5);
            \path let \p1 = (P46) in coordinate[vtx] (X46) at (\x1,-1.5);
            \path let \p1 = (P56) in coordinate[vtx] (X56) at (\x1,-1.5);

            \draw[name path=baseright, black] (8,-1.5) -- node[below] {$S_{1}'$} (X46) -- node[below] {$S_{2}'$} (X45) -- node[below] {$S_{3}'$} (X56) -- node[below] {$S_{4}'$} (13,-1.5);
            \node[right] at (13,-1.5) {$\mathfrak{B}$};

            \foreach \i/\j in {1/2,1/3,2/3} {
                \coordinate[vtx] (Q\i\j) at (P\i\j);
            }
            \node[above] at (P12) {$x$};
            \node[below right] at (P13) {$y$};
            \node[below left] at (P23) {$z$};

            \foreach \i/\j in {4/5,4/6,5/6} {
                \coordinate[vtx] (Q\i\j) at (P\i\j);
            }
            \node[above] at (P45) {$x'$};
            \node[above left] at (P46) {$y'$};
            \node[above right] at (P56) {$z'$};
            
            \foreach \x [count=\i] in {0,1.3,2} {
                \node[left] at (0,\x) {$\sigma_\i$};
            }

            \foreach \x [count=\j] in {2,0.7,0} {
                \node[right] at (13,\x) {$\sigma_{\j}$};
            }
        \end{tikzpicture}
    \centering
    \label{fig:R3-front}
\end{figure}
    In what follows we will use the notation of the above figure. When we refer to a specific marking, we will denote by \( [\gamma_{\sigma_i,j}] \) the marking of the strand colored \( \sigma_i \) over the region \( S_j, \) \( [\gamma_{\sigma_i,j}'] \) similarly. By convention, we take the strands in the above figure to be co-oriented upwards. Similarly, we will speak of a marking \( [\gamma_{\sigma}] \) as either oriented \emph{upwards} or \emph{downwards}, meaning, respectively, a marking realized by a path which travels from \( \sigma \) to \( \iota(\sigma) \) along the induced orientation of the fibers of \( pi, \) or against it.

    The continuation problem can be broken down into three broad cases depending on which two of the strands \( \sigma_i \) are matched by \( (\iota,\Gamma) \) over the region \( S_1. \) 

    We will start with the case \( \iota(\sigma_3 \cap \pi^{-1}(S_1)) = \sigma_2. \) In this situation, \( z \) is necessarily non-switching by the Maslov condition, and we subdivide further into the cases where zero, or one of \( x \) and \( y \) are switching. Note that both cannot switch by normality. When neither switches, it is straightforward to see that the ruling may be continued by one which is non-switching at all of \( x',y',z', \) and moreover that this is still normal (automatically), and preserves the number of switches (i.e. characteristic).
    
    When only \( x \) is switching, then the only possible continuations which match on the boundary are \( y',z' \) switching and \( x' \) non-switching, or \( x' \) switching, but neither of \( y',z'. \) Analysing the possible markings that allow \( x \) to switch, we see that, after traversing a non-switching \( y', \) either both \( [\gamma_{\sigma_1,2}'] \) and \( [\gamma_{\sigma_3,2}'] \) are oriented the same direction, and have covering linking number equal to 
    \[ \link([\tilde{\gamma}_{\sigma_1,1}],[\tilde{\gamma}_{\sigma_1,1}]) \]
    or are oppositely oriented, with \( [\gamma_{\sigma_1,2}'] \) pointing down and \( [\gamma_{\sigma_3,2}'] \) pointing up, thus having covering linking number 0. In either case, this yields a normal, characteristic continuation, while the alternative is usually not positive, and never characteristic, so we have our desired unique continuation.
    
    Finally when only \( y \) is switching there are again two possible continuations, but the only one which is characteristic is when only \( y' \) is switching. Similar to the above, in order for \( y \) to switch, either \( [\gamma_{\sigma_{2},1}]  \) and \( [\gamma_{\sigma_{1},1}] \) are both oriented the same way, and have covering linking 0, or the former is oriented upwards, and the latter downwards. In either case, we see that the conditions for \( y' \) to be a positive switch are satisfied, and the Maslov condition is automatic so this is a normal characteristic continuation, and is the unique such as already noted. 
    
    This concludes the case \( \iota(\sigma_3 \cap \pi^{-1}(S_1)) = \sigma_2. \) The remaining two are similar, and left to the reader.
\end{proof}

\subsection{Non-squeezing}\label{sec:3d-nonsqueezing}
With our two ruling counts in hand we are ready to tackle the non-squeezing problem. The main result of this section is the following elementary analog of \cref{thm:pushing-out}:
\begin{theorem}\label{thm:short-rulings-push-out}
    Suppose \( \Lambda \subset \widehat{D(1)} \subset J^1(\R;S^1) \) is a smoothly embedded Legendrian link contained in the interior of the pre-quantized disk of radius 1. Then 
    \begin{equation}\label{eqn:ruling-ineq} \#\!\mathrm{R}^{\C}(\Lambda) \leq \#\!\mathrm{R}^{\cyl}(\Lambda). \end{equation}
    In particular, if \( \Lambda \subset J^1(\R;S^1) \) violates \cref{eqn:ruling-ineq}, then \( \Lambda \) is not squeezable.
\end{theorem}
\begin{proof} We will show that, under the above embedding assumption, every normal disk ruling can be ``expanded'' to a normal circular ruling, and that distinct disk rulings remain distinct after expansion. 

We first describe the expansion procedure. Let \( \Lambda \subset J^1(\R;S^1) \) be as in \cref{thm:short-rulings-push-out} (and be equipped with a Maslov potential), and suppose \( \iota \in \mathrm{R}^{\C}(\Lambda) \) is a positive, Maslov pseudo-involution. As in \cite{ChekanovPushkar05}, \( \iota \) decomposes \( p(\Lambda) \) into a collection of ``eyes:'' continuously embedded circles, each with exactly 2 cusps, which meet only at crossings of \( p(\Lambda), \) such that the two points of an eye in a given fiber of \( \pi \) are always paired by \( \iota. \)

The key observation is then that, if \( x,y \in \Lambda_{\mathfrak{S}} \) are paired by \( \iota, \) \( \iota(x) = y, \) the angular separation of the points of the corresponding eye changes according to the difference between the \emph{slopes} of the front at \( x,y, \) and moreover, on either end of the eye, this angular separation must be 0, since the two strands must meet in a cusp. These slopes are exactly the cotangent coordinate in the jet, and we take the fiber of \( \pi \) to have length 1, so since \( \Lambda \subset \widehat{D(1)}, \) we conclude that no eye of \( \iota \) can ever have angular separation \( \pi \) (i.e. antipodal points are never matched by \( \iota \)). It follows that over each point of \( \mathfrak{B}, \) there is a continuously varying canonical short homotopy class of path between the points of a given eye of \( \iota, \) which, by normality, exactly satisfy the conditions of an admissible marking as described in \cref{def:circular-ruling}. Denote this marking by \( \Gamma_{\iota}. \)

Then \( \iota \mapsto (\iota,\Gamma_{\iota}) \) defines the required map \( \mathrm{R}^{\C}(\Lambda) \rightarrow \mathrm{R}^{\cyl}(\Lambda). \) It remains to check that distinct disk rulings remain distinct, but this is immediate, as the data of \( \iota \) (regarded now as only a weak pseudo-involution) is preserved. Thus our mapping is injective, and the claimed inequality follows.
\end{proof}

The simplest example of a non-squeezable Legendrian knot is the lift of an embedded loop in \( T^*\R \) enclosing area 1, a front for which is drawn below:
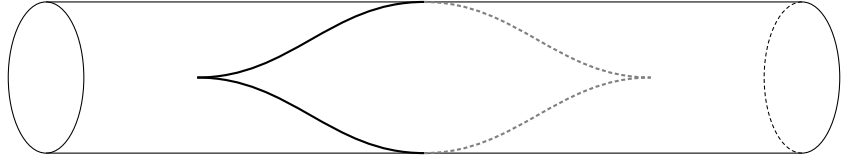
\begin{center}
    \begin{tikzpicture}[mydash/.style={dashed,dash pattern=on 1.5pt off 1pt}]
        \draw (0,0) ellipse (0.5 and 1);
        \draw (10,-1) arc(-90:90:0.5 and 1);
        \draw[mydash] (10,1) arc(90:270:0.5 and 1);
        \draw (0,-1) -- (10,-1);
        \draw (0,1) -- (10,1);
        \draw[thick] (2,0) to[out=0,in=180] (5,1);
        \draw[thick] (5,-1) to[out=180, in=0] (2,0);
        \draw[thick,mydash,gray] (5,1) to[out=0,in=180] (8,0);
        \draw[thick,mydash,gray] (5,-1) to[out=0,in=180] (8,0);
    \end{tikzpicture}
\end{center}
This admits exactly one disk ruling, which simply pairs the strands in each fiber, and no circular ruling, since any marking cannot be trivial at both ends. By passing to a cover, a lift of any loop enclosing integral area can be made to be of this form, and so we conclude:
\begin{corollary}\label{cor:elem-unknot-nonsqueezing}
    Let \( L \subset T^*\R^n \) is a smoothly embedded loop bounding area \( k \in \N, \) and \( \Lambda \subset J^1(\R;S^1) \) a Legendrian lift. Then \( \Lambda \) does not admit a Legendrian isotopy into the interior of \( \widehat{D}(k). \)
\end{corollary}
This is the 3-dimensional case of \cref{thm:high-d-squeezing}. The remainder of this article is devoted the high dimensional version.

\section{Categorical generalities}\label{sec:preliminaries}
In this section we'll collect some more terminology to be used throughout the remainder of the article. Fix a (graded) commutative (graded) ring \( R. \)

We denote by \( C(R) \) the \( R \)-linear abelian category of unbounded chain complexes of \( R \) modules, with morphisms given by chain maps. We will assume our categories are small. 

\subsection{Differential graded categories}
We briefly recall the essentials of dg-categories, and their homotopy theory. For details we refer to \cite{Toen-dg-lectures,Keller-dgcats,Toen-Morita}. A (small) \emph{dg-category} \( \mathcal{C} \) over \( R \) consists of the data of
\begin{itemize}
    \item A set of objects \( \ob(\mathcal{C}). \)
    \item For each pair \( x,y \in \ob(\mathcal{C}) \) an object of \( C(R), \) \( \hom(x,y). \)
    \item For each triple \( x,y,z \in \ob(\mathcal{C}) \) a \emph{composition} (\( C(R) \)-)morphism
    \[ \mu_{x,y,z}: \hom(x,y) \otimes \hom(y,z) \rightarrow \hom(y,z) \]
    \item For each \( x \in \ob(\mathcal{C}) \) a \emph{unit} morphism \( e_x:R \rightarrow \hom(x,x). \)
\end{itemize}
Such that the diagrams
\begin{center}
\begin{tikzpicture}
    \node (A) at (0,2) {$\hom(x,y) \otimes \hom(y,z) \otimes \hom(z,w)$};
    \node (B) at (8,2) {$\hom(x,z) \otimes \hom(z,w)$};
    \node (C) at (0,0) {$\hom(x,y) \otimes \hom(y,w)$};
    \node (D) at (8,0) {$\hom(x,w)$};
    \draw[->] (A) to node[above] {$\mu_{x,y,z} \otimes \id$} (B);
    \draw[->] (A) -- node[left] {$\id \otimes \mu_{y,z,w}$} (C);
    \draw[->] (B) -- node[right] {$\mu_{x,z,w}$} (D);
    \draw[->] (C) -- node[below] {$\mu_{x,y,w}$} (D);
\end{tikzpicture}
\end{center}
and
\begin{center}
    \begin{tikzpicture}
        \node (A) at (0,0) {$\hom(x,y) \simeq R \otimes \hom(x,y)$};
        \node (B) at (6,0) {$\hom(x,x) \otimes \hom(x,y)$};
        \node (C) at (12,0) {$\hom(x,y)$};
        \node (D) at (0,-2) {$\hom(x,y) \simeq \hom(x,y) \otimes R$};
        \node (E) at (6,-2) {$\hom(x,y)\otimes\hom(y,y)$};
        \node (F) at (12,-2) {$\hom(x,y)$};
        \draw[->] (A) -- node[above] {$e_x \otimes \id$} (B);
        \draw[->] (B) -- node[above] {$\mu_{x,x,y}$} (C);
        \draw[->] (A) to [bend right=10] node[below] {$\id$} (C);
        \draw[->] (D) -- node[above] {$\id \otimes e_y$} (E);
        \draw[->] (E) -- node[above] {$\mu_{x,y,y}$} (F);
        \draw[->] (D) to [bend right=10] node[below] {$\id$} (F);
    \end{tikzpicture}
\end{center}
commute for any \( x,y,z,w \in \ob(\mathcal{C}). \) In the sequel, we will leave the choice of ground ring implicit, and simply write \emph{dg-category}.

\begin{example}
    The category \( \mathbf{C}(R) \) whose objects are chain complexes of \( R \)-modules, and whose morphisms are the internal homs of chain complexes is a dg-category.
\end{example}

\subsubsection{Functors and modules}
Let \( \mathcal{C},\mathcal{D} \) be dg-categories. A dg-functor \( F:\mathcal{C} \rightarrow \mathcal{D} \) is the data of a map \( F:\ob(\mathcal{C}) \rightarrow \ob(\mathcal{D}), \) and morphisms of complexes \( F_{x,y}:\hom(x,y) \rightarrow \hom(F(x),F(y)), \) for every \(x,y\in\ob(\mathcal{C}),\) which are compatible with the unit and composition data of \( \mathcal{C} \) and \( \mathcal{D}. \) Such functors may be composed, and we denote the category with objects dg-categories, and morphisms dg-functors by \( \dgcat. \)

To a dg-category \( \mathcal{C}, \) there is a canonically associated \( R \)-linear category \( [\mathcal{C}] \) called its \emph{homotopy category}, whose objects are the same as those of \( \mathcal{C}, \) and morphisms are
\[ \hom_{[\mathcal{C}]}(x,y) := H^0(\hom(x,y)). \]
Every dg-functor induces a(n ordinary) functor \( [F]:[\mathcal{C}] \rightarrow [\mathcal{D}] \) (i.e. the operation \( [\cdot] \) is a functor from \( \dgcat \) to \( \cat. \)).

A dg-functor \( F \) is called \emph{quasi-fully-faithful} if \( F_{x,y} \) is quasi-isomorphism for every \( x,y \in \ob(\mathcal{C}). \) \( F \) is called \emph{quasi-essentially surjective} if \( [F] \) is essentially surjective. \( F \) is called a \emph{quasi-equivalence} if it is both quasi-fully-faithful and quasi-essentially surjective. 

Given dg-functors \( F,G:\mathcal{C} \rightarrow \mathcal{D} \) a dg-natural-transformation \( \psi:F \rightarrow G \) of degree \( d \) is a natural transformation composed of morphisms \( \psi(x):F(x) \rightarrow G(x) \) which lie in \( \hom^d_{\mathcal{D}}(F(x),G(x)). \) Composition is defined by composition in \( \mathcal{D}. \)

The dg-natural-transformations can be assembled into a complex \( \hom(F,G), \) whose degree \( d \) elements are the degree \( d \) natural transformations, and differential is given by \( (d\psi)(x) := d_{\mathcal{D}}(\psi(x)). \) From this perspective, composition induces a chain map
\[ \hom(F,G) \otimes \hom(G,H) \rightarrow \hom(F,H). \]
In particular, the collection of all dg-functors \( \mathcal{C} \rightarrow \mathcal{D} \) can be assembled into a dg-category \( \fun_{dg}(\mathcal{C},\mathcal{D}) \) whose morphism complexes and composition are defined as above.

The dg-functors to chain complexes are of particular importance:
\begin{definition}\label{def:dg-modules}
    A (left) dg-module \( \mathcal{M} \) over a dg-category \( \mathcal{C} \) is a dg-functor
    \[ \mathcal{M}: \mathcal{C}^{op} \rightarrow \mathbf{C}(R). \]
    In this special case, we denote the category \( \fun_{dg}(\mathcal{C},\mathbf{C}(R)) \) by \( \mathcal{C}\!-\!\mods. \)
\end{definition}

\begin{example}
    The most important examples of dg-modules over a dg-category \( \mathcal{C} \) are provided by the \emph{Yoneda} modules induced by the objects of \( \mathcal{C}: \)
    \[ \mathcal{Y}^l_{x}(y) := \hom_{\mathcal{C}}(y,x) \qquad \mathcal{Y}^r_{x}(y) := \hom_{\mathcal{C}}(x,y) \]
    are, respectively, left and right dg-modules over \( \mathcal{C}. \) Together these assemble into the dg-Yoneda embeddings
    \[ \mathcal{Y}^l: \mathcal{C} \rightarrow \mathcal{C}\!-\!\mods \qquad \mathcal{Y}^r: \mathcal{C}^{op} \rightarrow \mathcal{C}^{op}\!-\!\mods. \]
\end{example}

There is a model structure (see \cite{Tabuada-models}) on \( \dgcat \) whose weak equivalences are given by quasi-equivalences, and whose fibrations are the dg-functors \( F:\mathcal{C} \rightarrow \mathcal{D} \) such that
\begin{itemize}
    \item \( F_{x,y} \) is an epimorphism of \( R \)-complexes for every \( x,y \in \ob(\mathcal{C}). \)
    \item Given objects \( x \in \ob(\mathcal{C}), \) \( z \in \ob(\mathcal{D}) \) and a morphism \( a:F(x) \rightarrow z \) such that \( [a] \) is an isomorphism in \( [\mathcal{D}], \) there is an object \( y \in \ob(\mathcal{C}) \) and a morphism \( b:x \rightarrow y, \) which is an isomorphism in the homotopy category, such that \( F(y) = z \) and \( F(b) = a. \)
\end{itemize}
We call this the \emph{quasi-equivalence} model structure, and denote by \( [\dgcat]^{qe} \) the corresponding homotopy category.

\subsubsection{Derived categories}
Throughout the remainder of the paper, we will usually work with dg-derived categories. We collect here some more relevant facts and definitions.

The (ordinary) derived category of a ring \( R \) is classically defined as the homotopy category of chain complexes of \( R \)-modules equipped with the model structure whose weak equivalences are quasi-isomorphisms, and fibrations are provided by term-wise projective complexes. This can be viewed as a localization of the category of complexes of \( R \)-modules, along the quasi-isomorphisms, or, equivalently, a quotient by the full sub-category of acyclic complexes (in the sense of Verdier). This last approach was applied to dg-categories by Drinfeld:
\begin{theorem*}[\cite{drinfeld-quotients}]
Let \( \mathcal{A} \) be a dg-category, and \( \mathcal{B} \subset \mathcal{A} \) be a full dg-subcategory. Then there is a dg-category \( \mathcal{A} \rightarrow \mathcal{A}/\mathcal{B}, \) characterized by the property that the induced map of sets
\[ \hom_{[\dgcat^{qe}]}(\mathcal{A}/\mathcal{B},\mathcal{C}) \rightarrow \hom_{[\dgcat^{qe}]}^{\mathcal{B}}(\mathcal{A},\mathcal{C}) \]
is a bijection for all \( \mathcal{C} \)\footnote{This is a more categorical reformulation of Drinfeld's result due to Tabuada in \cite{Tabuada-dg-quotient}}. Here \( \hom_{[\dgcat^{qe}]}^{\mathcal{B}}(\mathcal{A},\mathcal{C}) \) denotes the morphisms which send all objects of \( \mathcal{B} \) to objects of \( \mathcal{C} \) with acyclic endomorphisms.
\end{theorem*}
Applying this to the dg-category of complexes of \( R \)-modules \( \mathbf{C}(R), \) and the full sub-category of acyclic complexes yields the dg-enhancement of the derived category of \( R, \) \( \mods_R. \) 

More generally if \( \mathcal{C} \) is a dg-category, we define its \emph{dg-derived category} \( \der\mathcal{C} \) to be the quotient of \( \mathcal{C}\!-\!\mods \) by the full subcategory of functors which land in acyclic complexes. Equivalently, by the universal property above, this is \( \fun_{dg}(\mathcal{C},\mods_R). \) 

There is a natural functor \( \mathcal{C} \rightarrow \der\mathcal{C} \) obtained by composing the (left) Yoneda embedding with the localization \( \mathcal{C}\!-\!\mods \rightarrow \der\mathcal{C}. \) We denote by \( \mathcal{C}^{tri} \) the full dg-subcategory of \( \der\cc \) spanned by finite pushouts of the image of \( \mathcal{C} \) under the derived Yoneda embedding. 
\begin{definition}\label{def:perfect-and-proper}
    The dg-category of \emph{perfect} modules over a dg-category \( \cc, \) \( \perf(\cc) \) is the full subcategory of \( \der\cc \) spanned by retracts of objects of \( \cc^{tri}. \)
    
    The dg-category of \emph{proper} modules over \( \cc, \) \( \prop(\cc) \) is defined as the full subcategory of \( \der\cc \) spanned by those modules which land in \( \perf(R). \)
\end{definition}

\subsubsection{The Morita model structure}
There is another model structure on \( \dgcat, \) which, in some sense, is more natural than the quasi-equivalence model structure from the point of view of homotopy theory. 
\begin{theorem*}[\cite{Toen-Morita,Tabuada07}]
    \( \dgcat \) carries the structure of a (combinatorial) model category whose cofibrations are those of the quasi-equivalence model structure, and whose weak equivalences are those dg-functors \( F:\mathcal{C} \rightarrow \mathcal{D} \) which induce quasi-equivalences \( \perf(\mathcal{C}) \rightarrow \perf(\mathcal{D}). \) 
\end{theorem*}
We call this the \emph{Morita} model structure on \( \dgcat. \) 

\begin{remark}
The above statement is difficult to find (with any accompanying proof) in these terms, so we give some further explanation: To\"en constructed the internal hom of dg-categories (with their monoidal structure) in \cite{Toen-Morita}. He described the internal hom objects in terms of certain (quasi-)representable bimodules, and showed that the (simplicial) nerve of this category computes the mapping spaces of the Dwyer--Kan localization of \( \dgcat. \) He also showed that the Yoneda embedding induces a quasi-equivalence between a dg-category \( \mathcal{C} \) and its category of quasi-representable modules. From this, it follows that the derived Yoneda embedding into \( \perf(\cc) \) is a Morita weak equivalence by definition, and moreover yields a fibrant replacement. After this replacement, idempotent completeness of the perfect modules implies that every quasi-representable bimodule over \( \perf(\cc) \) is induced by a genuine dg-functor, and so the above yields the expected morphisms on the level of homotopy categories.
\end{remark}

The main result of \cite{Cohn16} shows that nerve (in the sense of \cite{HA,HTT}) of \( \dgcat, \) equipped with the Morita model structure, is equivalent to a certain subcategory of stable \( \infty \)-categories, an equivalence which will be leveraged later. This will be explained more in \cref{sec:categorical-comparison}, and for now can be regarded as some loose justification of our comment at the top of this section.

As observed by Tabuada in \cite{Tabuada-dg-quotient}, Drinfeld's quotient is still the correct notion of localization when working in the Morita model structure on dg-categories:
\begin{theorem*}[\cite{Tabuada-dg-quotient}*{Theorem 4.0.5}]
    The Drinfeld quotient induces an isomorphism
    \[ \hom_{\dgcatmor}(\mathcal{A}/\mathcal{B},\mathcal{C}) \rightarrow \hom_{\dgcatmor}^{\mathcal{B}}(\mathcal{A},\mathcal{C}). \]
    As above, the second morphism set consists of those represented by bimodules which are acyclic on objects of \( \mathcal{B}. \)
\end{theorem*}
We will not make direct use of this result, but it is perhaps nice to keep in mind when considering \( \infty \)-categorical quotients later.

\subsubsection{Periodic dg-categories}\label{sec:periodic-dg}
Later, we would like to work with objects that carry only a \( \Z/2 \) grading, rather than a \( \Z \) grading, as the dg-categories we have considered until now do. The remedy for this is as follows: Given a (honest) commutative ring \( R, \) let \( R_2 := R[u,u^{-1}] \) denote the ring of graded Laurent polynomials over \( R, \) with the convention that \( u \) lies in degree 2. 
\begin{definition}\label{def:2-periodic-dg}
    A 2-periodic dg-category over \( R \) is a dg-category over \( R_2. \)
\end{definition}
We have been deliberately agnostic to our coefficients thus far so much of what we have set up above may be applied directly. In particular, we are principally concerned with passing between the category of dg-categories, and some category of \( \infty \)-categories, and thus with obtaining a model for the morphism spaces of the Dwyer--Kan simplicial localization (\cite{DwyerKan80}) of \( \dgcat_{R_2} \) along its quasi-equivalences. It is shown in \cite{Dyckerhoff11}*{\S 5} that the constructions of \cite{Toen-Morita} can be carried over to the 2-periodic setting, and thus that the Morita model structure is well defined, and still realizes Dwyer--Kan localization in this generality. Similarly, Drinfeld's quotient can be imported directly. This suffices to show the desired equivalence, as we will discuss further in \cref{sec:categorical-comparison}.

\subsubsection{Hereditary categories}\label{sec:hereditary-cats}
Given a commutative ring \( R \) concentrated in degree 0, we may regard an \( R \)-linear abelian category \(\mathcal{C}\) as an \( R \) dg-category with morphisms concentrated in degree 0. The dg-derived category \( \der\cc \) then agrees with the usual notion of derived category of an abelian category. By tensoring with \( R_2 \) before passing to the derived category, we similarly obtain a dg-model for the 2-periodic derived category of \( \cc. \) There is a convenient criterion for determining when this operation is purely formal:
\begin{definition}
    An \( R \)-linear abelian category \( \mathcal{C} \) is called \textit{hereditary} if the (bi-)functors \( \ext^k(-,-) \) all vanish for \( k \geq 2. \)
\end{definition}
Note that this is equivalent to the full subcategory spanned by the image of \( \cc \) in the derived category \( \der\cc \) having morphisms concentrated in degrees 0 and 1. 

Hereditary categories have the property that objects of \( \perf(\cc) \) (resp. \( \perf_2(\cc) \)) are isomorphic to their cohomology, i.e. to a finite direct sum of shifts of objects of \( \cc \). The proof of this fact is a standard lifting argument:
\begin{lemma}\label{lem:periodic-inheritance}
    Let \( \mathcal{C} \) be a hereditary category, \( A^{\bullet} \) a 2-periodic \( \mathcal{C} \) complex with perfect endomorphisms. Then \( A^{\bullet} \) is isomorphic to it's cohomology in \( \der\mathcal{C}. \) In particular, \( A^{\bullet} \) is quasi-isomorphic to a finite direct sum of shifts of objects of \( \mathcal{C}. \)
\end{lemma}

\subsection{\( \infty \)-Categories}\label{sec:infty-cats}
We now recall some general facts from the theory of \( \infty \)-categories that will be useful later. Throughout we will work in the setup of \cite{HTT,HA}. 

Our principal use for \( \infty \)-categories is as a convenient world to carry out homotopical constructions. In particular, we will later be concerned with sheaves of dg-categories, and our main theorem is proved by constructing objects of global sections glued back from a cover. In the process of this construction, one generally needs to make many choices, and verify that those choices are coherent (up to all higher homotopies). In the dg-setup, this quickly becomes quite laborious, as many replacements must be taken (according to the model structure) in order to, for example, invert equivalences. It is not a-priori obvious that certain constructions one wishes to perform are well defined (or even make sense). The \( \infty \)-categorical viewpoint clarifies many of these points, and provides a convenient language (and toolbox) for discussing the existence (and homotopical uniqueness) of coherent choices.

Throughout we will work with presentable, stable, symmetric monoidal \( \infty \)-categories. We leave the definitions of these terms to \cite{HTT,HA}, since they appear only as language in which to state the results of others justifying our free passage between the dg and \( \infty \)-categorical world, since all of our hands-on work will be undertaken in the dg-setting.

\subsubsection{Commas}
We will have two distinct use cases for \( \infty \)-categorical constructions. The first is in guaranteeing homotopy coherent replacements of general objects by those with a more convenient presentation. The precise setup is as follows: 

Let \( \mathcal{C} \) be a \( \infty \)-category, and \( \iota:\mathcal{D} \hookrightarrow \mathcal{C} \) be a full subcategory such that the induced inclusion on homotopy categories \( h\mathcal{D} \rightarrow h\mathcal{C} \) is an equivalence of category. 

Now let \( I \) denote the undirected interval object in \( \icat, \) i.e. the nerve of the ordinary category with two objects, and exactly one morphism between each pair of objects (necessarily all isomorphisms), and \( \Delta^1 \) the (nerve of the) 1-simplex. The \emph{iso-comma category} \( (\mathcal{C}\sim \mathcal{D}) \) is defined to be the limit in \( \icat \):
\begin{center}
    \begin{tikzpicture}
        \node (a) at (0,2) {$(\mathcal{C}\sim\mathcal{D})$};
        \node (b) at (3,2) {$\fun(I,\mathcal{C})$};
        \node (c) at (0,0) {$\mathcal{C} \times \mathcal{D}$};
        \node (d) at (3,0) {\( \mathcal{C} \times \mathcal{C} \)};
        \draw[->] (a) -- (b);
        \draw[->] (a) -- (c);
        \draw[->] (b) -- node[right] {$(F(0),F(1))$} (d);
        \draw[->] (c) -- node[above] {$(\id,\iota)$} (d);
    \end{tikzpicture}
\end{center}
On a 1-categorical level, this can be regarded as the collection of isomorphisms pairs \( c \rightarrow d \rightarrow c \) which compose to identity. 1-morphisms are then given by pairs of arrows \( c\rightarrow c',d\rightarrow d' \) such that the induced diagram commutes. 

Now consider the map \( (\mathcal{C} \sim \mathcal{D}) \rightarrow \mathcal{C} \) given by composing the upper right corner of the square with the projection \( \mathcal{C} \times \mathcal{C} \rightarrow \mathcal{C} \) to the first factor. Expanding upon the 1-categorical description above, we see that this projection is a Kan fibration, and an equivalence, since we require that the arrows \( c\rightarrow d \) are equivalences, and the inclusion \( \iota \) induce equivalence on homotopy categories. In particular, this is an acyclic Kan fibration, and thus has non-empty, contractible space of sections (equivalently, the space of homotopy inverses to the inclusion \( \iota \) is non-empty and contractible). We will make use of this construction later in \cref{sec:essential-rulings}.

\subsubsection{Sheaves}\label{sec:infty-sheaves}
The other direct use is as a setting for discussing sheaves. We will consider sheaves in \( \icat \) on the open site of some topological space. All such spaces which appear in the present article are paracompact (in particular, they may be regarded as embedded in some finite dimensional manifold), and so we do not distinguish between sheaves and hypersheaves (\cite{HTT}). The essential facts that we will make use of are the existence of a good notion of sheafification for pre-sheaves of \( \infty \)-categories, and six functor formalism (\cite{HTT}). We will use the fact that sheaves in \( \icat \) satisfy \v{C}ech descent (by definition), and that isomorphisms may be checked on stalks. 

Let us also record here a convenient fact (\cite{NadlerShende-weinstein}*{\S6}): Suppose \( \mathcal{F} \) is a pre-sheaf in \( \icat \) on some paracompact space \( X \), whose stalks are presentable, stable, and computed by some local cofinal sequences whose restrictions stabilize to identity. Then the sheafification \( \mathcal{F}^+ \) takes values in the full subcategory of presentable \( \infty \)-categories. Moreover, if \( \mathcal{F} \) takes value in the stable \( \infty \)-categories (and restriction is exact) the sheafification is also stable, and agrees with the sheafification in the \( \infty \)-category of presentable stable \( \infty \)-categories. 

\subsection{From dg-categories to \( \infty \)-Categories}\label{sec:categorical-comparison}
We now conclude the discussion we started in \cref{sec:periodic-dg}. What remains is the introduce the \( \infty \)-categorical object that lives on the other side of the equivalence of \cite{Cohn16}, and argue that the equivalence extends to dg-categories over a graded ring.

As discussed above, the Morita model structure can be constructed in the 2-periodic dg-setting, and its still realizes the Dwyer--Kan localization of \( \dgcat_{R_2}. \) The other side of our equivalence is given by the \(\infty\)-category of \( R_2 \)-linear presentable stable \( \infty \)-categories.

For a general \( E_{\infty} \) ring spectrum \( S, \) these are, by definition, the modules in presentable stable \( \infty \)-categories over the category of \( S \)-modules, i.e. presentable, stable \( \infty \)-categories tensored over \( S-\mods \) (in the sense of \cite{HA}). Now \cite{HA}*{Theorem 7.1.2.13} shows that the unbounded derived (\(\infty\)-)category of the graded ring \( R_2 \) is naturally equivalent to the category of module spectra over the periodic ring spectrum \( HR_2 := \bigvee_{n\in\Z} \sigma^{2n}HR \) (\(HR\) denotes the Eilenberg-MacLane spectrum associated to \( R. \)). This induces a monoidal equivalence 
\[ \der R_{2}-\mods \cong HR_2-\mods, \]
and thus
\[ \icat^{st,pr}_{\der R_{2}-\mods} \cong \icat^{st,pr}_{HR_2-\mods} \]
of the associated categories of linear categories. Then the main theorem of \cite{Haugseng-rectification} shows that \( \icat^{st,pr}_{\der R_{2}-\mods} \) is exactly (the nerve of) the category of \( R_2 \) dg-categories, i.e. the Dwyer--Kan localization of \( \dgcat_{R_2}. \) In particular, homotopy limits (thus sheaves) in \( \dgcatmor_{R_2} \) are induced by limits in the \( \infty \)-category of presentable, stable, \( HR_2 \)-linear \( \infty \)-categories, thus (under the cofinality condition described in \cref{sec:infty-sheaves}) \( \icat. \) As such we will treat the two settings as interchangeable; computations are done in the dg-setting, and the results are fed into \( \infty \)-categorical descent machinery in order to obtain our main technical results.

\section{Microsheaves}\label{sec:microsheaves}
The invariants we will use to detect non-squeezing phenomena in high dimensions are some variants of the category of \emph{microsheaves} supported on a Legendrian submanifold of a contact manifold \( \Lambda \subset (Y,\xi) \). We will take the approach of \cite{NadlerShende-weinstein} to the general definition of microsheaves on a Weinstein manifold, and recall here some generalities regarding the construction, and the properties we will make use of later.

For concreteness, we fix a field \( \kk \), and let \( \mathcal{C} := \vect_{\kk_2} \) denote the dg-derived (\( \infty \)-)category of 2-periodic complexes of \( \kk \)-vector spaces (for much of this section any symmetric monoidal stable presentable \( \infty \)-category will do).

Associated to a smooth manifold \( M \) is the (\( \infty \)-)category \( \sh(M) \) of \( \cc \) sheaves on \( M. \) Many properties of \( \sh(M) \) are developed in \cite{KashiwaraSchapira} in the setting where \( \mathcal{C} \) is the \emph{bounded} (classical) derived category of a commutative ring. These can (mostly\footnote{As remarked in \cite{NadlerShende-weinstein}*{\S 6}, there are situations in which the unbounded setting is genuinely different, but they do not appear in the present context.}) be carried over directly to the unbounded, dg-derived setting via the machinery of \cite{Spaltenstein-unbounded} to pass to the unbounded case, using the fact that the universal property characterizing Drinfeld's quotient readily implies that the homotopy category of the dg-derived category of a ring is precisely the classical derived category in the sense of \cite{KashiwaraSchapira}. A proof of the essential lemma (and some deeper discussion of this extension) can be found in \cite{RobaloSchapira-lemma}.

Accordingly, an object \( \mathcal{F} \in \sh(M) \) has a notion of \emph{microsupport} (see \cite{KashiwaraSchapira}*{Chapter 5}) in the cotangent bundle to \( M, \) defined as the (closure of the) set of \( (q,p) \in T^*M \) such that there exists a neighborhood \( U \) of \( q, \) and a smooth function \( f:U \rightarrow \R, \) with \( f(q) = 0, \) \( df(q) = p \) such that the mapping cone of the restriction
\[ \mathcal{F}(f^{-1}(\R_{<\epsilon})) \rightarrow \mathcal{F}(f^{-1}(\R_{<-\epsilon})) \]
is homologically non-trivial for all \( \epsilon > 0. \) As a consequence of the definition, the microsupport is always conical (with respect to the standard Liouville structure on \( T^*M \)), and by \cite{KashiwaraSchapira} it is also a co-isotropic subset. Given a sheaf \( \mathcal{F} \in \sh(M) \) we will denote it's microsupport by \( \musupp(\mathcal{F}). \) Note that if \( \mathcal{F} \) is locally constant, then \( \musupp(\mathcal{F}) = \supp(\mathcal{F}) \subset M \) is contained in the 0-section of \( T^*M. \)

Given a closed subset \( K \subset T^*M \) there is a full subcategory of \( \sh(M) \) defined as the category of sheaves with microsupport contained in \( K. \) We denote this by \( \sh_{K}(M). \)

The situation we will primarily be concerned with is the case where \( K \) is conic Lagrangian, smooth away from positive codimension. In this case, at smooth points of \( K, \) the cone defining microsupport can (up to a shift) be functorially associated to a sheaf \( \mathcal{F} \) microsupported in \( K \) (\cite{KashiwaraSchapira}*{Proposition 7.5.3}). Following \cite{GPS3} we call this object the \emph{microstalk} of \( \mathcal{F} \) at \( (q,p). \) When \( p = 0, \) the shift vanishes, and the microstalk simply recovers the usual stalk of \( \mathcal{F} \) at \( q. \)

The key insight relevant to the construction at hand is that the category of sheaves on \( M \) itself sheafifies (``microlocalizes'') over \( T^*M: \) Given an open \( U \subset T^*M \) there is an \( \infty \)-category
\[ \mush^{pre}(U) := \sh(M)/\sh_{M\setminus U}(M). \]
Here (and throughout) the quotient denotes \( \infty \)-categorical localization, equivalently Drinfeld's quotient when \( \cc \) is the derived category of some commutative ring. By the properties of the quotient, if \( V \subset U \) there is a restriction (defined by further localization) \( \mush^{pre}(U) \rightarrow \mush^{pre}(V). \) In particular \( \mush^{pre} \) defines a pre-sheaf of (stable) \( \infty \)-categories on \( T^*M. \)

\begin{remark}
    We could perform the same construction in the restricted microsupport setting, i.e. fixing a priori a closed subset \( K \). In the case where \( K \) is a Lagrangian conic, the categories involved are also presentable.
\end{remark}

Given such an object, we can form its sheafification in \( \icat \)\footnote{As discussed in \cref{sec:infty-sheaves}, in the case of restriction to conic Lagrangian microsupport all categories involved become presentable, and so all reasonable notions of sheafification agree. See \cite{NadlerShende-weinstein}*{Remark 6.1} for additional discussion.}, and obtain a genuine sheaf of stable \(\infty\)-categories which we denote by \( \mush. \) 

The term ``microlocalization'' of \( \sh(M) \) over \( T^*M \) is justified by the results of \cite{KashiwaraSchapira}, where it is shown that \( \mush(T^*M) \cong \sh(M). \)\footnote{Strictly speaking these objects are not discussed in \cite{KashiwaraSchapira}, only the stalks of the pre-sheaf \( \mush^{pre} \) appear. Nevertheless, the arguments contained therein suffice to prove this (cf. \cite{NadlerShende-weinstein}*{\S7.1}).} 

Note that, by conicity of the microsupport, \( \mush \) is more intrinsically defined with respect to the conic topology on \( T^*M \) (i.e. that generated by all Liouville invariant opens). Eventually, our constructions will take place primarily in the cosphere bundle \( S^*M, \) where \( \mush \) is also sensible (and the usual topology is induced by the conic topology, corresponding to conic open subsets disjoint from the 0-section). This restriction is also denoted by \( \mush. \)

Given an open (without loss of generality conic) \( U \subset T^*M \) the objects of the category \( \mush(U) \) have a well defined notion of microsupport detected by non-triviality in the stalk at a point \( (x,\xi) \in U. \) In fact, it follows from the construction of the \( \muhom \) functor (\cite{KashiwaraSchapira}), and the comparison therein with the homs of the stalks of \( \mush^{pre}, \) thus \( \mush, \) that this is inherited directly from the notion of microsupport in \( \sh(M). \) 

As a consequence, given a closed, conic coisotropic \( K \subset T^*M \) (or simply a closed co-isotropic in \( S^*M \)) we can speak of the subsheaf \( \mush_K \) whose value on an open set \( U \) consists of the full subcategory of \( \mush(U) \) of objects microsupported in \( K \cap U. \) In fact, this object localizes completely to \( K \): \( \mush_K \) is isomorphic to the pushforward of a sheaf of categories on \( K. \) As before, if \( (q,p) \in U \) is a smooth conic Lagrangian (resp. Legendrian) point of \( K, \) \( \mathcal{F} \in \ob(\mush_{K}(U)) \) we can functorially associate a microstalk to \( \mathcal{F}. \) 

The essential property that makes \( \mush_K \) a useful object for our purposes is contact invariance: Let \( K \subset S^*M \) be closed, \( \phi_t:S^*M \rightarrow S^*M \) be a contact isotopy. Such an isotopy (regarded as conic on \( T^*M \)) induces a Lagrangian in \( T^*M\times T^*M \times T^*\R, \) whose projection to the second two factors yields a map from subsets of \( T^*M \) to subsets of \( T^*M\times \R. \) Applying conicity, we obtain a similar map from subsets of \( S^*M \) to those of \( S^*M \times T^*\R. \) In each case this map takes \( K \) to a subset \( \Phi(K) \) (which we call the \emph{trace} of \( K \) under \( \phi_t \)) with the property that symplectic (coisotropic) reduction of the fiber of \( \Phi(K) \) over \( t \in \R \) is exactly \( \phi_t(K). \)
\begin{proposition}[\cite{NadlerShende-weinstein}*{Lemma 6.6}]
    Suppose \( K \subset S^*M \) closed, and \( \phi_t:S^*M \rightarrow S^*M \) is a contact isotopy. Then there is a canonical isomorphism induced by pushforward
    \[ \left(\phi_t\right)_*\mush_{K} \cong \mush_{\phi_t(K)}. \]
\end{proposition}
This is a consequence of a general stabilization principle: 
\begin{proposition}[\cite{NadlerShende-weinstein}*{Lemma 6.3}]
    Suppose (the germ of) \( K \subset S^*M \) is contactomorphic to the germ of some \( k \times N \subset U \times T^*N \) (\( U \) open, contact). Then (under this contactomorphism) \( \mush_K \) is locally constant along \( N. \)
\end{proposition}
In particular, when \( N = \R^n \) is contractible, \( \mush_K\res{k \times \set*{*}} \) determines \( \mush_K \) by pullback.  Equivalently, microsheaves are locally determined by the corresponding category over a transverse coisotropic reduction of \( K. \) 

\subsection{Microsheaves on stably polarized Weinstein manifolds}
To now we have discussed microsheaves as an object associated to the cotangent (or cosphere) bundle of a smooth manifold, but we would like to work in the more general setting of (stably polarized) Weinstein symplectic manifolds. This passage has been accomplished in \cite{NadlerShende-weinstein,Shende-hprinciple}, and we recall here the construction, and results which will be of relevance to us later. 

Let \( (W,d\lambda) \) be a Weinstein manifold, and let \( \tau \) be a (stable) polarization of \( W. \) As observed in \cite{Shende-hprinciple} Gromov's h-principle for open symplectic embeddings provides an abundance of \emph{exact} embeddings of \( W \) into \( S^*\R^{N} \) for large \( N, \) the space of which can be made arbitrarily connected by further increasing \( N \) (in particular, any pair of embeddings may be assumed to be isotopic). Exactness of the embedding means that we ask for the ambient contact form to restrict to the chosen Liouville form for \( W. \) In particular, we can extend to a contact embedding of the \( \R \) contactization of \( W. \)

The data of the stable polarization now enters as a means of \emph{thickening} this high codimension embedding of the contactization: The data of a stable polarization of \( W \) is equivalent to the data of a section of the Lagrangian Grassmannian of the \emph{symplectic normal bundle} to the associated contactization (see \cite{NadlerShende-weinstein}*{Lemma 9.1}), which is precisely a section of the Lagrangian Grassmannian of the normal bundle to our contact embedding of \( W\times \R. \)

Now to a Legendrian \( \Lambda \subset W\times \R \) we can associate a canonical local Legendrian thickening \( \Lambda^{\tau} \subset  S^*\R^N \) by extending \( \Lambda \) by the Lagrangian sub-bundle of the normal bundle over \( \Lambda \) specified by \( \tau, \) and taking the image of this in a symplectic tubular neighborhood. This descends from the global thickening of the contactization \( (W\times \R)^{\tau} \) constructed similarly. 

Using this object we define the \emph{\( \tau \) twisted microsheaves on \( W \times \R \)} to be the sheaf on \( W\times \R \)
\[ \mush_{W\times\R}^{\tau} := \mush_{(W\times \R)^{\tau}}\res{W\times \R}. \]
This object is rather larger than what we wanted; an invariant associated directly to the Weinstein manifold \( (W,d\lambda). \) This can be recovered by restricting further to the objects microsupported in the \emph{skeleton} of \( W, \) \( \mathfrak{sk}(W), \) defined as the inward limit of the Liouville flow, and realized as its Legendrian lift to the contactization contained in the level \( W \times 0. \) That is, we define
\[ \mush_{W}^{\tau} := \mush_{\sk(W)^{\tau}}\res{W}. \]
More generally, if \( \Lambda \subset \partial_{\infty}W \) is a smooth Legendrian, \( \mathfrak{c}(\Lambda) \) its Liouville cone, we define the twisted microsheaves on the Weinstein pair \( (W,\Lambda) \) to be the sheaf
\[ \mush_{(W,\Lambda)}^{\tau} := \mush_{\sk(W)^{\tau} \cup \Lambda^{\tau}}\res{W}. \]

The invariance results described in the previous section, together with Gromov's embedding h-principle, combine to show that this is well defined (i.e. an invariant of only \( (W,\lambda,\tau) \)). The potential dependence on the choice of Weinstein primitive can also be relaxed:
\begin{theorem*}[\cite{NadlerShende-weinstein}*{Theorem 9.14}]
    Suppose \( \lambda,\lambda' \) are two Weinstein primitives for the symplectic manifold \( (W,\omega), \) which moreover are homotopic through a primitive \( \eta \) on \( W \times T^*\R \) which induces a skeleton stratifiable by isotropics. Then there is an equivalence 
    \[ \mush_{(W,d\lambda)}^{\tau} \cong \mush_{(W,d\lambda')}^{\tau} \]
    induced by this cobordism.
\end{theorem*}

\subsection{Polarizations and Maslov data}\label{sec:sheaf-polarizations}
There is still a question of how \( \mush^{\tau}_{W} \) depends on the polarization \( \tau. \) This is addressed in detail \cite{NadlerShende-weinstein}*{\S11}, and we refer there for details (similar discussion can be found in \cite{jin-jhomomorphism}). In order to state the relevant result we must introduce slightly more notation. Given a stable presentable symmetric monoidal \( \infty \)-category \( \cc \) we denote by \( \mathrm{Pic}(\cc) \) the spectrum of automorphisms of \( \cc \) regarded as a module over itself (for further discussion see \cite{ABGHR14}).

The upshot is the following:
\begin{theorem*}[\cite{NadlerShende-weinstein}*{\S11}]
    There is a canonical map \( W \rightarrow B(U/O) \rightarrow B^2\mathrm{Pic}(\cc) \) where the first arrow classifies the stable Lagrangian Grassmann bundle, a null-homotopy of which determines a category of microsheaves in \( \cc \) on \( W. \) In particular, two stable polarizations \( \tau,\tau'\) (regarded as null-homotopies of the first map) induce equivalent microsheaves if they agree upon further composition to \( B\mathrm{Pic}(\cc). \)
\end{theorem*}
\begin{remark}
    Here we regard a stable polarization as a null-homotopy of a classifying map, i.e. a trivialization of the corresponding principal bundle. The map \( W \rightarrow B^2\mathrm{Pic}(\cc) \) should then be regarded as classifying a \( B\mathrm{Pic}(\cc) \) bundle over \( W, \) and so by looping, the sections \( \tau,\tau' \) push forward to sections of this \( B\mathrm{Pic}(\cc) \) bundle, i.e. trivializations thereof. Since this bundle is necessarily trivializable in order to define microsheaves, it follows that the possible twistings of microsheaves with coefficients in fixed category \( \cc \) are a torsor over \( [W,B\mathrm{Pic}(\cc)], \) the homotopy classes of maps \( W \) to a delooping of the Picard spectrum of \( \cc. \) In general, this is a very complicated object, but often some low height truncations are computable, and so if the topology of \( W \) is sufficiently simple one can obtain some handle on the different isomorphism classes of twistings. We will apply this later in \cref{sec:morse-objects} to the particular cases relevant to the non-squeezing story at hand.
\end{remark}

\section{The invariants}\label{sec:sheaf-invariants}
We will now specialize to the case \( W = T^*\R^n, \) and \( \cc = \vect_{\kk_2}, \) the 2-periodic dg-derived category of vector spaces over some finite field \( \kk. \) Moreover we fix once-and-for-all (iso-)contact embeddings \( \iota_{S^1}:\widehat{W} \rightarrow \partial_{\infty}(T^*S^1 \times W) \) and \( \iota_{\C}:\widehat{W} \rightarrow \C \times W \) as the part of the boundary with positive \( T^*S^1 \) coordinate, and the complement of the binding of the trivial open book respectively.

Let \( \Lambda \subset \widehat{W} \) be a smooth Legendrian submanifold.

\subsection{The circular category}\label{sec:circular-sheaves}
The first invariant we associate to \( \Lambda \) is microsheaves on \( (S^1 \times \R^n,\iota_{S^1}(\Lambda)), \) polarized by the fiber. We ask that the stalk (microstalk along the 0-section) vanishes away from the support of \( \pi(\Lambda) \). We denote this category by \( \mush^{\cyl}_{\Lambda}. \) This is canonically equivalent to the microlocalization of the \( \cc \) sheaves on \( S^1\times \R^n \) (micro)supported on the Liouville cone over \( \Lambda \) union the 0-section, and we will occasionally take this perspective without comment. When we wish to refer to global objects we will drop the prefix \( \mu \). Often, we will want to refer to the microlocal category over a subset of the bifurcation space. To reduce clutter, we denote this by the prefix \( \m \). Observe that since \( \pi \) is a smooth, proper, submersion, this itself is a sheaf of dg-categories, without re-sheafifying.

\subsection{The disk category}\label{sec:disk-sheaves}
Similarly, we will consider the \( \cc \) microsheaves on the relative Weinstein pair \( (\C\times \R^n,\iota_{\C}(\Lambda)), \) with its canonical polarization\footnote{Since the ambient space is contractible, there is exactly one homotopy class of (stable) polarization.}, which we denote by \( \mush^{\C}_{\Lambda}. \) As above the prefix \( \m \) will be used to denote the pushforward to the bifurcation space, and no prefix denotes the global sections. In this case, this coincides with the ambient Weinstein manifold, but does enforce that one sees the entire \( \Lambda \) cone over a given point. Again, this itself forms a sheaf of categories over the bifurcation space without correction.

\subsection{The Morse category}\label{sec:morse-objects}
As a means of communicating between these two settings, we will sometimes consider a category of sheaves on \( \R \times \R^n, \) with microsupport on some Legendrian \( \Lambda \subset \partial_{\infty}^{+} T^*\R^{n+1} \). We will denote this category by \( \sh^{\R}_{\Lambda}. \) 

We will sometimes (but not always) insist that the stalk should vanish near \( \pm\infty \) in the first \( \R \) factor. We denote the full subcategory of objects satisfying this by \( \sh^{\R}_{\Lambda,0}. \) In dimension 3, this is exactly the category which by now has been thoroughly studied in, e.g., \cite{STZ-knots,augmentations-are-sheaves}. As before, the prefix \( \mu \) denotes the corresponding microsheaves.

Let us collect here several observations which will be useful later:

First, if \( \Lambda \subset I \times \R^n \subset \R \times \R^n \) a smooth embedding \( \phi:I \times \R^{n} \rightarrow S^1 \times \R^n \) induces a fully faithful pushforward \( \phi_*:\sh^{\R}_{\Lambda,0} \rightarrow \sh^{\cyl}_{\phi(\Lambda)} \) onto the full subcategory of circular objects microsupported on the image of \( \Lambda \) with vanishing stalk outside the image of \( \phi. \) This functor depends on \( \phi \) only up to the symplectic isotopy class relative \( \Lambda \) of the induced map on cotangent bundles. In particular, any such \( \phi,\psi \) with the same restriction to \( p(\Lambda) \) induce the same functor up to a contractible choice.

On the other hand, suppose \( \Lambda \subset \widehat{T^*\R^n} \subset \partial_{\infty}\C\times T^*\R^n \) has front projection in the complement of a horizontal hyperplane \( H \). Now consider the Legendrian \( \Lambda \sqcup H^{\pm \epsilon} \), consisting of \( \Lambda \) together with two small (Reeb) push-off's of \( H \). Now the category \( \sh^{\C}_{\Lambda} \) embeds in \( \sh^{\C}_{\Lambda\sqcup H^{\pm \epsilon}} \) as precisely the objects with vanishing microstalk along \( H^{\pm \epsilon}. \) This category can, moreover, be identified with \( \sh^{\R}_{\Lambda,0}, \) where \( \Lambda \) is embedded by thinking of the region from \( H^{\epsilon} \) to \( H^{-\epsilon} \) as the positive end, and the complementary region as the negative end. 

Call a Legendrian \( \Lambda \subset \widehat{T^*\R^n} \) \textit{bounded} if, outside of some compact region of \( F \) the front \( p(\Lambda) \) is a disjoint union of horizontal hyperplanes. 
\begin{proposition}\label{prop:double-stopped-disks}
    Let \( \Lambda \subset \widehat{T^*\R^n} \) be a bounded front generic Legendrian, \( \mathcal{F} \in \sh^{\C}_{\Lambda} \) and let \( U \subset F \setminus p(\Lambda) \) be a connected component of the complement of \( p(\Lambda) \) with surjective bifurcation projection. This data induces a fully faithful equivalence
    \[ \sh^{\C}_{\Lambda} \rightarrow \sh^{\R}_{\Lambda,0}, \]
    which depends only on the component \( U \).
\end{proposition}

\begin{proof}
    Either \( \Lambda \) is empty, in which case the proposition is tautological, or \( \pi:U \rightarrow \mathfrak{B} \) is a smooth fiber bundle with contractible fiber (diffeomorphic to the open interval). Thus there exists a (smooth) section \( f \), and the space of such is itself contractible. Let \( H^{\pm}_{f} \subset U \) be small push-offs of the image of \( f, \) disjoint from each other, without loss of generality \( H^+ > H^- \) with respect to the fiber orientation induced by that on \( S^1. \)

    The pair \( (\C \times T^*\R^n, \Lambda \sqcup H^{\pm}) \) is a relative Weinstein pair in the sense of \cite{NadlerShende-weinstein}. Straightening \( H^{\pm} \) and positioning them antipodally via fibered frontal isotopy yields a (canonical up to contractible choice) identification with
    \[ \sh^{\C}_{\Lambda' \sqcup \pm \frac{\pi}{2} \times \R^n}. \]
    Now deforming the standard Liouville structure \( xdy - ydx \) on \( \C \) by \( txy \) we obtain a relative Weinstein cobordism
    \[ (\C \times T^*\R^n, \Lambda \sqcup H^{\pm}) \rightleftarrows (T^*\R \times T^*\R^n, \Lambda) \]
    yielding an equivalence 
    \[ \begin{tikzcd} \sh^{\C}_{\Lambda \sqcup H^{\pm}} \arrow[r,"\phi_H"] & \sh^{\R}_{\Lambda}.\end{tikzcd} \]
    Here the copy of \( \Lambda \) on the right hand side is regarded as being embedded in the positive boundary at \( \infty, \) and the zero section is identified with the cone on \( H^{\pm}. \)

    On the other hand the inclusion of relative skeleta induced by \( \Lambda \hookrightarrow \Lambda \sqcup H^{\pm} \) induces a functor by pullback
    \[ \begin{tikzcd} \mods\left( \sh^{\C}_{\Lambda} \right) \arrow[r,"\psi_H"] & \mods\left( \sh^{\C}(\Lambda \sqcup H^{\pm}) \right). \end{tikzcd} \]
    The desired functor is the composition of \( \psi_H \) and \( \phi_{H}^*, \) the pullback on modules induced by \( \phi_H. \) 

    It remains to verify that this is an equivalence. Note that \( \phi^H \) takes (co-representatives of) the microstalks at \( \Lambda \) onto their geometric counterparts in Morse category, and (corepresentatives of) the microstalks at \( H^{\pm} \) to the stalk at \( \pm\infty. \). In particular, the full subcategory of objects with vanishing microstalk on \( H^{\pm} \) is mapped fully faithfully, and essentially surjectively, onto the full subcategory \( \sh^{\R}_{\Lambda,0} \) by \( \phi^*_H. \) Moreover, by its construction as a localization along \( H^{\pm}, \) \( \psi_H \) is a fully faithful embedding with image the objects with vanishing microstalk on \( H^{\pm}. \) It follows that the composition is an equivalence as desired.

    The equivalence constructed in this way depends a priori on our choice of \( H^{\pm}, \) but different choices of these induce canonical fully faithful equivalences (via a contractible choice of frontal isotopy)
    \[ \sh^{\C}_{\Lambda \sqcup H^{\pm}} \cong \sh^{\C}_{\Lambda \sqcup (H')^{\pm}} \]
    with respect to which the functors \( \phi,\psi \) are natural. The proposition follows.
\end{proof}
We will later refer to the procedure yielding this equivalence as \textit{double stop insertion} or \textit{double stopping}.

Putting these together we have a means of lifting a disk object supported on a local Legendrian to a circular object on the same, which may be thought of as ``installing a 0-section,'' and ``excising an interval''  (double stopping in \( U \) and then pushing forward along inclusion).

\begin{remark}
    It follows from the geometric nature of the construction of the double stopped lift that it is compatible with the restriction maps of the sheaf of categories \( \m\sh^{\bullet}. \) In particular, if \( \Lambda \) is a smooth Legendrian with front avoiding the graphical hypersurface \( H, \) over \( U \subset \mathfrak{B}, \) and \( V \subset U \) is open then the following diagram commutes:
    \begin{equation*}
        \begin{tikzcd}
            \m\sh^{\C}_{\Lambda}(U) \arrow{r}{\DS_H} \arrow{d}{\mid_V} & \m\sh^{\cyl}_{\Lambda}(U) \arrow{d}{\mid_V}\\
            \m\sh^{\C}_{\Lambda}(V) \arrow{r}{\DS_H} & \m\sh^{\cyl}_{\Lambda}(V)
        \end{tikzcd}
    \end{equation*}
    where \( \DS_H \) denotes the embedding induced by double stopping at \( H. \)
\end{remark}

\begin{remark}
    Later, we will want to apply the above construction to multiple objects of \( \sh^{\C}_{\Lambda} \) simultaneously, but with respect to potentially different choices of double stop, depending on the particular object considered. In particular we would like to make sense of all of these lifts landing in the same target. If we take the vertical (stable) polarization of the \( \C \) factor, and try to work with respect to the fiber polarization on \( T^*S^1 \) we will run into issues producing compatible lifts on a geometric level: double stopping at antipodal points will not yield a global polarization that agrees with the fiber, but rather an oriented twist (see \cref{sec:polarizations}) of this.

    In order to resolve this, we should work with respect to this non-standard polarization of the cotangent bundle throughout. On the other hand, if we work over 2-periodic complexes, the Picard group (\( \pi_0 \) of the Picard spectrum) is \( \Z/2, \) and our twist is oriented, thus a square. Then the Maslov descent arguments of \cite{NadlerShende-weinstein}*{\S 11} show that the microsheaves determined by a stable polarization depend only on the induced map to \( B\mathrm{Pic}(\mathcal{C}) \), which has \( \pi_1 \cong \Z/2 \) by the above. Since our space has the homotopy type of \( S^1, \) we then conclude that the 2-periodic microsheaves induced by the fiber polarization, and this twisted polarization are isomorphic. Consequently, when we want to compute in \( \sh^{\cyl}_{\Lambda}, \) we will continue to work in the fiber polarization, and tacitly pass through this identification when we wish to compare.
\end{remark}

\section{Local models}\label{sec:combinatorics}
We now describe some ``combinatorial'' models for the categories \( \sh^{\bullet}_{\Lambda}(\widehat{T^*\R^{n}}). \) In dimension 3, these will allow us to give a complete description of (microlocal rank 1) objects of \( \sh^{\bullet}_{\Lambda} \) in terms of the front projection for front generic \( \Lambda. \) More generally, these allow us to produce a version of ``normal ruling'' (\cite{ChekanovPushkar05}) for Legendrians in all dimensions, and will help facilitate the constructions of \ref{sec:essential-rulings}. 

In dimension 1 there is exactly one closed, connected contact manifold, \( S^1, \) which can be thought of as the contactization of a point. In this case, Legendrians are collections of points on \( S^1, \) and the two fillings we consider are \( T^*S^1 \) and \( \C, \) where we think of our Legendrian living at \( +\infty \) in \( T^*S^1. \) Thus the bifurcation space is a point. These form the fundamental local model for sheaves on Legendrians in higher dimension, as this computes the local category over the open strata of \( \mathfrak{B} \) via symplectic reduction along the fiber of \( \Pi. \). We also give a complete local description of the codimension 1 front singularities, which allows a ``by hand'' construction of sheaves on Legendrians in dimension 3. Finally, we give a partial classification of local objects near front singularities occurring in higher codimension, which will play an important role in the constructions of \cref{sec:essential-rulings}.

\subsection{Positively marked cylinders}\label{sec:cylindrical-cat}
We begin with the case of \( T^*S^1 \) equipped with some collection of points \( \Lambda = \set*{\theta_i} \subset S^1 \cong \partial_{+,\infty}T^*S^1. \) As is by now standard in the business of microlocal sheaves, \( \Lambda \) may be regarded as inducing a stratification of \( S^1. \) In this situation, the geometry is sufficiently simple, and is in fact a triangulation, so we may apply the results of \cite{Nadler-branes}*{\S 2}. From this observation we obtain an equivalent characterization of \( \sh_{\Lambda} \) as a certain subcategory of the proper modules over the poset category induced by the stratification.

This allows us to give a characterization of the local category, but for later computations it will be desirable to give a completely standard algebraic model from which the geometry has been abstracted. This equivalence can be rephrased in the following terms:
\begin{proposition}\label{thm:circular-quiver}
\[ \sh^{\cyl}_{\Lambda} \cong \prop(\hat{\mathcal{A}}_n), \]
where \( n = \left| \Lambda \right|, \) \( \prop(\hat{\mathcal{A}}_n) \) denotes the locally perfect, 2-periodic derived representations of the cyclic quiver with \( n \) vertices.
\end{proposition}
\begin{proof}
As noted above, there is an equivalence
\[ \sh^{\cyl}_{\Lambda} \rightarrow \prop^+\left( \boldsymbol{\Lambda}^{op}, \vect_{\kk_2} \right), \]
where the right hand side is the dg-category of functors from the stratification induced by \( \Lambda \) to the 2-periodic dg-derived category of \( \kk \) vector spaces, such that downward generalization is (quasi-)isomorphism, and all complexes have cohomology of finite rank. Similarly, if \( \left| \Lambda \right| = n \) there is an equivalence 
\[ \prop(\hat{\mathcal{A}}_n) \rightarrow \fun^+\left(\boldsymbol{\Lambda}, \vect_{\kk_2} \right), \]
which is induced by identifying \( \hat{\mathcal{A}}_n \) with the dual graph to cell structure induced by \( \boldsymbol{\Lambda}, \) where the value over a vertex of \( \boldsymbol{\Lambda} \) is defined to be that at the downward vertex, with downward generalization given by the identity.
\end{proof}

\begin{remark}
    The category \( \prop(\hat{\mathcal{A}}_n) \) is itself canonically split, 
    \[ \prop(\hat{\mathcal{A}}_n) \cong \prop^{0}(\hat{\mathcal{A}}_n) \oplus \mathbf{L}_{\perf}(\hat{\mathcal{A}}_n), \]
    into the (homologically) nilpotent representations, and local systems of perfect complexes (all propagation maps are (quasi-)isomorphism). The latter factor is quite complicated (and presents the principle complication in classifying the representations of cyclic quivers) but for our purposes may be safely ignored, since the vanishing at \( \infty \) condition imposed on our microsheaves in dimension greater than 1 ensures that this is trivial. As such we will generally only address the full subcategory of nilpotent representations below.
\end{remark}

\subsubsection{Classification}\label{sec:circular-classification}
The proper, nilpotent dg-derived representations of \( \hat{\mathcal{A}}_n \) satisfy a Gabriel-type theorem. That is, every such representation is isomorphic to a unique, finite, direct sum of shifts of nilpotent indecomposables, each of which is a representation such that all but two propagation maps are isomorphism, and each of those have mapping cone of homological rank 1. In analogy to the ``barcode'' terminology often applied to the output of the usual Gabriel's theorem for the \( \mathcal{A}_n \) quiver, we will refer to this as a \textit{circular-bar decomposition}.

Remarkably, as explained in \cref{sec:hereditary-cats} it is enough to show that ordinary nilpotent \( \kk \)-representations of \( \hat{\mathcal{A}}_n \) have such a decomposition. This is not difficult, but we will include a description of the classification (with proof), since it has been rather difficult to extract from the literature in plain language. Let \( \mathcal{C} := \rep^{0}_{fin}(\hat{\mathcal{A}}_n) \) denote the (\( \kk \)-linear abelian) category of finite rank nilpotent representations of \( \hat{\mathcal{A}}_n. \) Let \( \mathcal{B} \) denote the full subcategory of objects with total cone of rank 2. The classical version is the following:
\begin{proposition}\label{prop:semi-simple-circles}
    \( \mathcal{B} \) consists exactly of the indecomposables of \( \mathcal{C}. \) In particular, every object of \( \mathcal{C} \) is isomorphic to a unique direct sum of objects of \( \mathcal{B}. \)
\end{proposition}

\begin{proof}
    Let \( \rho \in \mathcal{C} \) be a finite rank, nilpotent representation. Label the vertices of \( \hat{\mathcal{A}}_n \) \( 0,1,\ldots,n-1 \) cyclically ordered, and let \( p_i:i \rightarrow i+1 \) denote the elementary propagation morphism connecting \( i \) to \( i+1 \). By nilpotence, for each \( v \in \rho(i) \) there is some \( l \in \N \) such that
    \[ \rho\left( p_{i+l-1}\circ p_{i+l-2} \circ \cdots \circ p_i \right)(v) = 0. \]
    This induces a filtration of each \( \rho(i) \) by the minimal such \( l \) to annihilate a given element, called the \textit{length filtration}. Denote by \( \rho^l(i) \) the subspace annihilated by the \( l \)-fold composition of propagation. Clearly \( \rho^l \) with the induced propagation maps forms a subrepresentation.

    Let \( N \) be the largest number such that \( \rho^N \hookrightarrow \rho \) is not an isomorphism. Choose a complement to \( \rho^{N}(i) \) over each vertex. Denote by \( \psi^N \) the subrepresentation induced by the image of this complement under propagation, and by \( \psi^N_i \subset \psi^N \) the further subrepresentation in the image of the complement chosen over the vertex \( i. \) Considering the filtration, we necessarily have \( \psi^N_i \cap \psi^N_{j} = 0 \) whenever \( i \neq j. \) Thus the map
    \[ \bigoplus_{i} \psi^N_i \rightarrow \psi^N \]
    induced by inclusion is injective and surjective, thus isomorphism.

    Similarly, a choice of basis for the complement over \( i \) presents \( \psi^N_i \) as a direct sum of objects of \( \mathcal{B} \).

    Now proceeding down in the filtration, we can choose complements to \( (\rho^{N-1} + \psi^N \cap \rho^N) \) inside \( \rho^N, \) inducing a subrepresentation \( \psi^{N-1} \) as before. By construction the sum of inclusions
    \[ \psi^N \oplus \psi^{N-1} \rightarrow \rho \]
    is injective. Thus proceeding inductively, we present \( \rho \) as 
    \[ \rho \cong \bigoplus_{m = 0}^{N} \psi^m \]
    where some \( \psi^m \) may be trivial.

    The decomposition thus produced is evidently into indecomposables, and the above procedure applied to an object of \( \mathcal{B} \) produces a unique chain of maximal length which necessarily surjects.
\end{proof}

Since nilpotent representations of the cyclic quiver are hereditary (they are, in particular, finite representations of the infinite linear quiver), we immediately have
\begin{corollary}\label{cor:circular-decomposition}
    Every object of \( \prop^{0}(\hat{\mathcal{A}}_n) \) has a canonical circular-type-bar decomposition, i.e. is (quasi-)isomorphic to a  direct sum of bar-objects. In particular, the inclusion of the full subcategory spanned by direct sums of shifted bars is an equivalence.
\end{corollary}

It is clear from the proof that objects of \( \mathcal{B} \) are classified by their endpoints (where their propagation maps have cokernel or kernel), and the number of cyclic compositions required to annihilate them. 

The equivalence of \cref{thm:circular-quiver} yields such a decomposition for objects \( \mathcal{F} \in \sh^{\cyl}_{\Lambda} \) of nilpotent type, which we will also refer to as the circular bar decomposition of \( \mathcal{F}. \)

\subsection{Marked disks}\label{sec:marked-disks}
Next we consider \( \C, \) again with some collection of marked points at \( \infty. \) A direct computation of a dg-realization of this category would be slightly cumbersome, so we will instead extract what we need from the fully faithful embedding of \cref{prop:double-stopped-disks}. From this we can describe everything in terms of some convenient, simple objects: call an object \( b \in \sh^{\C}_{\Lambda} \) a \textit{disk-type-bar} if \( b \) has non-trivial microstalk at exactly two (distinct) points, each of rank 1, with endomorphisms of rank 1. The disk-type-bars are classified by their endpoints, and the grading of their microstalks. Working over \( \kk_2 \) there are exactly two non-isomorphic bar objects for each pair of distinct markings. 

We would like to compare objects of the disk category to those of the circular category by means of the classification given by \cref{sec:circular-classification}. To this end, we have the following result:
\begin{theorem}\label{thm:disk-pairings}
    Let \( \mathcal{F} \in \sh^{\C}_{\Lambda} \). Then \( \mathcal{F} \) is isomorphic to a finite direct sum of bars:
    \[ \mathcal{F} \cong \bigoplus_{i = 1}^{N} \left(b_{l_i,k_i}^{\sigma(i)}\right)^{m_i} \]
    where \( l_i,k_i \) index Legendrian markings, and \( \sigma(i) \) denotes grading data. 
\end{theorem}

\begin{proof}
    The functor of \cref{prop:double-stopped-disks} yields a non-canonical fully faithful equivalence to \( \sh_0(\R,\Lambda). \) Gabriel's theorem and hereditary then imply both existence (directly) and uniqueness: Morphisms between bars can be computed in the Morse category, the existence of two distinct decompositions violates Gabriel's theorem after pullback.
\end{proof}

Let us emphasize an immediate corollary of this uniqueness statement: the disk-type-bar decomposition obtained via local double stop insertion as in \cref{prop:double-stopped-disks} is independent of \textit{where} the double stop is inserted.

\subsection{Codimension 1}\label{sec:codim-1}
For computational purposes later, it will be convenient to understand how each of our categories behave near the simplest types of front singularities we will encounter: simple, transverse crossings of two sheets of the front, and folds. 

\subsubsection{Disk sheaves}\label{sec:codim-1-disks}
To deal with disk objects, we employ the same trick as before: Since singularities are local in the front, there is always some region where we can cut the front along a hyperplane. It follows that the fold and crossing rules are then the same as those for the induced local Morse-type objects (see \cite{STZ-knots} for proof). 

In particular, over a small bifurcation neighborhood of a crossing
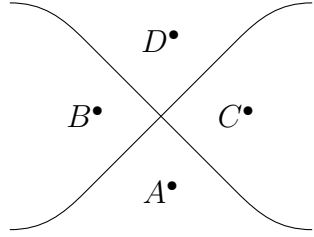
\begin{center}
    \begin{tikzpicture}
        \draw (-1,0.5) to[out=0,in=135] (0,0) to (1,-1) to (2,-2) to[out=-45,in=180] (3,-2.5);
        \draw (-1,-2.5) to[out=0,in=-135] (0,-2) to (1,-1) to (2,0) to[out=45,in=180] (3,0.5);
        \node (A) at (1,0) {$D^{\bullet}$};
        \node (B) at (1,-2) {$A^{\bullet}$};
        \node (C) at (0,-1) {$B^{\bullet}$};
        \node (D) at (2,-1) {$C^{\bullet}$};
    \end{tikzpicture}
\end{center}
where \( A^{\bullet},B^{\bullet},C^{\bullet},D^{\bullet} \) denote the stalks in the four regions cut out by the crossing, in order to be microsupported on the Legendrian with front this diagram we must have that the total complex (where maps are given by upwards generalization)
\[ A^{\bullet} \rightarrow B^{\bullet} \oplus C^{\bullet} \rightarrow D^{\bullet} \]
is acyclic.

Similarly near a fold
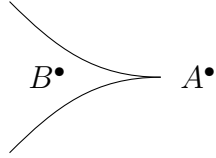
\begin{center}
    \begin{tikzpicture}
        \draw (0,0) to[out=-45,in=180] (2,-1);
        \draw (0,-2) to[out=45,in=180] (2,-1);
        \node (A) at (2.5,-1) {$A^{\bullet}$};
        \node (B) at (0.5,-1) {$B^{\bullet}$};
    \end{tikzpicture}
\end{center}
the map \( B^{\bullet} \rightarrow A^{\bullet} \rightarrow B^{\bullet} \) must factorize the identity.

Let us emphasize that this description is given \emph{having fixed a double stop} in order to pass to the Morse picture. In particular, the notion of stalk of a disk sheaf in a region in the front is not well defined until we fix such a basepoint.

\begin{remark}
    Strictly speaking in each case above the assertion is that, upon restriction to a small neighborhood of this diagram, every global sheaf is (quasi)isomorphic to one of the above forms. Since we work locally in the bifurcation space anyway, we are deliberately obscuring this subtlety.
\end{remark}

\subsubsection{Circular sheaves}\label{sec:codim-1-circles}
Circular sheaves cannot be so directly dismissed, but the approach is essentially the same (and by now standard). In this setting we are dealing with a front generic Legendrian in a cotangent bundle, so we can understand the (local) category of sheaves in terms of sheaves constructible with respect to a regular stratification of \( \Lambda_{loc} \). From this perspective, the local computation is identical to the usual (\( \R^2 \)) case, and we again arrive a the local crossing and fold conditions described above.

This said there is a crucial difference between the circular and disk cases hidden in the trick we applied to reduce the disk case to \( \R^2. \) By adding extra stops and thus distinguishing a basepoint, we are able to understand the local picture, but we've also broken some cyclic symmetry: for example, in order for a pairing to die in a fold, all that is required in the disk case is that the relevant strands are paired as they come into the fold, in any way. On the other hand, there are objects supported away from the fold, pairing the relevant strands in the circular case which do not extend over it. There are similar complications occur in the case of crossings. In sum, the local pictures are extremely similar, but have enough subtle differences to produce drastic changes in the global picture.

\begin{remark}\label{rem:ruling-by-sheaves}
    The upshot of the above discussions is a means of associating a ruling, in the sense of \cref{sec:cylindrical-rulings} to a front generic Legendrian given a microlocal-rank-1 sheaf supported on it (i.e. object of either \( \sh^{\C,1}_{\Lambda} \) or \( \sh^{\cyl,1}_{\Lambda} \)). In particular, at smooth points of the front symplectic reduction brings us to the sheaf category of either a marked disk, or positively marked cylinder. Then the canonical bar decomposition yields a pairing of the strands of the Legendrian over that point, perhaps with the added data of a path from one point to the other in the circular case, provided by the support of the stalk of a given bar. The fold and crossing conditions above ensure that the positive, Maslov conditions are imposed, though, in principle, there may be additional constraints imposed around more complicated front singularities.
\end{remark}

\subsection{Singular bar decompositions}\label{sec:singular-bars}
While giving complete combinatorial descriptions of sheaves near arbitrary front singularities is intractable, we can still give some reasonable combinatorial presentation of the local behavior: 
Let \( \Lambda \) be a Legendrian in \( \partial^{+} T^{*}X \times T^*\R^n, \) where \( X \in \set*{S^1,\R} \) and \( \mathcal{F} \in \sh^{\bullet}_{\Lambda}, \) where \( \bullet \in \set*{\cyl,\R} \) depending on the setting. Suppose moreover that \( \Lambda \) is front generic, and let \( U \subset X \) a proper open subset such that every front singularity of \( \Lambda \) has projection in \( U \times \R^n. \) Then by Gabriel's theorem, or \cref{cor:circular-decomposition}, \( \mathcal{F} \) splits canonically into a collection of bar-type objects with microsupport in \( X\setminus U \times \R^n, \) and a sheaf \( \mathcal{F}_{sing} \) with the property that the propagation across \( U, \) 
\[ \mathcal{F}_{sing}(U \times \R^n) \rightarrow \mathcal{F}_{sing}(\partial_+U \times \R^n) \]
is 0. Precisely, we have:
\begin{proposition}\label{prop:singular-bar-decomp}
    Let \( \mathcal{F},U \) be as above, \( K \supset U \) a compact, contractible subset of \( X \) disjoint from \( p(\Lambda), \) and \( h:X \rightarrow X \) a smooth, proper, surjective map of fiber bundles, which is fiberwise constant on \( U\times \R^n, \) and a diffeomorphism on \( X \setminus K\times \R^n. \) Then there is an isomorphism
    \[ \mathcal{F} \cong \mathcal{F}_{sm} \oplus \mathcal{F}_{sing} \]
    where \( \mathcal{F}_{sm} \) has (positive) microsupport disjoint from \( K\times \R^n \) and \( h_*\mathcal{F}_{sing} \) has a bar decomposition consisting exclusively of objects microsupported on \( h(U\times \R^n) \).
\end{proposition}
We call such a splitting a \textit{singular bar decomposition}.

Since we work with the circular and Morse categories, locally in the front in this section, we will write \( \sh \) instead of \( \mush \) throughout, understood as the category of sheaves with singular support conditions, identified by the canonical equivalence discussed in \cref{sec:microsheaves}.

We will use the following lemma:
\begin{lemma}\label{lem:collapse-maps}
    Let \( c:M \rightarrow M \) be a smooth map, \( U \subset \bar{U} \subsetneq M \) be a contractible open subset such that \( c\mid_{M\setminus U} \) is a diffeomorphism. Let \( \sh_{M\setminus U}(M) \) denote the full subcategory of sheaves with microsupport contained in \( 0_M \cup T^*(M\setminus \bar{U}). \) Then the restriction of the (derived) direct image functor \( c_* \) to this subcategory is fully faithful.
\end{lemma}
Whose proof follows from \cite{KashiwaraSchapira}*{Prop.~2.7.8}, and (some microsupport estimates...), and a standard homological algebra fact characterizing direct summands in triangulated categories:
\begin{lemma}\label{lem:sum-iff-factorization}
    Let \( A,B \in \mathcal{C} \) a pre-triangulated dg-category. Then there is a splitting
    \[ B \cong A \oplus C \]
    if and only if there is a factorization of the identity \( 1_A \)
    \[ \begin{tikzcd} A \arrow{r}{\iota} \arrow[bend right=20,swap]{rr}{\id} & B \arrow{r}{\varpi} & A. \end{tikzcd} \]
\end{lemma}

\begin{proof}[Proof of \cref{prop:singular-bar-decomp}]
    We begin by fixing some more notation: Let \( \tilde{U} = U \times \R^n, \tilde{K} \) likewise, \( \Lambda_U := \Lambda \cap \tilde{U}, \) \( \Lambda_{sm} = \Lambda \setminus \Lambda_U, \) and let \( \Sigma \subset \tilde{K} \) be a smooth, graphical hypersurface disjoint from \( \Lambda \) with \( \Sigma >_{X} \Lambda_U \) in the fiberwise order induced on  \( \tilde{U} \) by the orientation of \( X \). Let
    \[ \sigma^*:\sh^{\bullet}_{\Lambda} \rightarrow \sh^{\bullet}_{\Lambda \sqcup \Sigma} \]
    denote the restriction of scalars functor defined by allowing microsupport along \( \Sigma, \) \( \sigma_* \) its left adjoint. Similarly, let \( f^*,f_* \) be the corresponding (left, right) adjoint pair defined via localizing along \( \Lambda_U \) (regarding \( \sh^{\bullet}_{\Lambda} \) as proper modules over the wrapped microsheaves). Similarly, let \( g^*,g_* \) denote the localizations along \( \Lambda_U \) defined on \( \sh_{\Lambda \cup \Sigma}. \) Note that \( f_*f^* \) is a fully faithful equivalence on \( \sh_{\Lambda_{sm}}, \) \( \sigma_*\sigma^* \) is equivalence on \( \sh_{\Lambda}, \) and \( g_*g^* \) is an equivalence on \( \sh_{\Lambda_{sm} \cup \Sigma} \) which preserves the full subcategory in the image of \( g_*\sigma^*f^*\sh_{\Lambda_{sm}}, \) of objects with microsupport contained in \( \Lambda_{sm}. \)

    We will abuse notation and write 
    \[ h_*: \sh^{\bullet}_{\diamond} \rightarrow \sh^{\bullet}_{h(\Lambda_{sm}) \sqcup h(\tilde{U})} \]
    for the direct image functor, where \( \diamond \) denotes any of the above subsets, and \( h^{-1} \) for its right adjoint. The target category of \( h_* \) is the same in each case; this follows from the singular support estimates of \cite{KashiwaraSchapira}*{Chapter~V}. Moreover, the target category consists of objects microsupported on a Legendrian with front projection a disjoint union of graphical hypersurfaces, so by isotopy invariance and stabilization the target category is (quasi-)equivalent to quiver representations. As such objects have canonical bar decompositions by one of \cref{cor:circular-decomposition} or Gabriel's theorem. In particular, there is a decomposition
    \[ h_*\mathcal{F} \cong \mathcal{G} \oplus \tilde{\mathcal{F}}_{sing}, \]
    with \( \mathcal{G} \) microsupported, away from \( h(\tilde{U}) \), and \( \tilde{\mathcal{F}}_{sing} \) as in the statement of the proposition. 

    Let us now record some additional relations between the various functors and sub-categories collected thus far. First, all of our operations (i.e. those induced by \( f,g,\sigma,h \)) act fully faithfully on \( \sh^{\bullet}_{\Lambda_{sm}} \) (regarded as the full subcategory of objects supported on \( \Lambda_{sm} \) or, equivalently, the image of this category under extension of scalars, or inverse image). Moreover, the actions of \( f,g,\sigma \) restricted to \( \sh^{\bullet}_{\Lambda_{sm}} \) commute with \( h_*, \) in the sense that, for example \( h_*g_* = h_*. \) In the case of \( f,g,\sigma \) this follows from anti-microlocalization (i.e. restricting microsupport is localization), and for \( h \) this follows from \cref{lem:collapse-maps}.

    Now let \( \tilde{\mathcal{G}} = h^{-1}\mathcal{G} \in \sh_{\Lambda_{sm}} \). By \cref{lem:sum-iff-factorization}, it suffices to show that there is a factorization of identity:
    \[ f^*\tilde{\mathcal{G}} \rightarrow \mathcal{F} \rightarrow f^*\tilde{\mathcal{G}} \]
    if and only if there is such a factorization
    \[ g_*\sigma^*f^*\mathcal{G} \rightarrow g_*\sigma^*\mathcal{F} \rightarrow g_*\sigma^*f^*\mathcal{G}. \]
    (Such a factorization always exists, since the image of \( g_* \) is equivalent to one of our standard categories of quiver representations by construction.)

    The ``only if'' forward direction is immediate by applying \( g_*\sigma^*. \) 
    
    For the converse, suppose we have such a factorization after localizing along \( \Lambda_{U}. \) By adjunction there is a natural isomorphism
    \[ \hom(\mathcal{F},f^*\tilde{\mathcal{G}}) \rightarrow \hom(g_*\sigma^*\mathcal{F},g_*\sigma^*f^*\tilde{\mathcal{G}}) \]
    since \( \sigma^* \) is a fully faithful embedding. On the other hand, the adjunction \( h^{-1},h_* \) gives an isomorphism
    \[ \hom(h^{-1}h_*f^*\tilde{\mathcal{G}},\mathcal{F}) \cong \hom(h_*f^*\tilde{\mathcal{G}},h_*\mathcal{F}). \]
    Let \( \tilde{\phi} \) denote the lift of \( \phi \) defined by this isomorphism.
    \( h_*,h^{-1} \) are actually an adjoint equivalence on \( \sh_{\Lambda_{sm}\cup \Sigma}, \) and so there is a canonical isomorphism
    \[ \hom(h^{-1}h_*f^*\tilde{\mathcal{G}},\mathcal{F}) \cong \hom(f^*\tilde{\mathcal{G}},\mathcal{F}). \]
    Let \( \tilde{\psi} \) denote the induced lift of \( \psi. \)

    Since \( h_*g_* = h_* \) on \( \sh^{\bullet}_{\Lambda_{sm}}, \) the lifts of the components of our factorization thus provided have the property that
    \[ g_*(\phi \circ \psi) = id. \]
    Then by full faithfulness of \( g_* \) on \( \sh^{\bullet}_{\Lambda_{sm}} \) the proposition follows.
\end{proof}

The upshot of this is a splitting of the objects of \( \mush^{\bullet}_{\Lambda} \) near a pointlike singularity for \( \bullet \in \set*{\C,\cyl} \) (the former by double stopping), which we will leverage in the next section.

\section{Small Legendrians}\label{sec:essential-rulings}
In this section we will prove \cref{thm:pushing-out}. Let \( \Lambda \subset \widehat{T^*\R^n} \) be a smooth, closed, front-generic Legendrian, and \( \mathcal{F} \in \sh^{\C,1}_{\Lambda} \) a micro-local rank 1 object supported there-on.

Before beginning recall that we have fixed an identification of \( \widehat{T^*\R^n} \) with \( T^*\R^n \times S^1 \), invariant under the Liouville flow on \( T^*\R^n \). In particular, we trivialize the \( S^1 \) factor over the base. We will now also fix a \textit{good cover} \( \mathbb{U} \) of \( \mathfrak{B} \) refining the cover by neighborhoods of strata induced by the singularity stratification\footnote{Recall that a good cover of a topological space is composed of contractible opens, all of whose finite intersections are also contractible. In the present setting, such a cover could be obtained by, for example, fixing a simplicial refinement of the singularity stratification}. Without loss of generality (replacing by a finer good cover if necessary), \( \mathbb{U} \) is partitioned:
\[ \mathbb{U} = \bigsqcup_{\mathfrak{x} \in \mathfrak{X}} \mathbb{U}_{\mathfrak{x}} \]
where \( \mathbb{U}_{\mathfrak{x}} \) consists of those elements of \( \mathbb{U} \) which intersect the stratum \( \mathfrak{x}, \) and no other stratum of equal or greater codimension. Moreover, we assume (without loss of generality by shrinking) that \( U \in \mathbb{U}_{\mathfrak{x}} \) intersects exactly the strata in the (open) star of \( \mathfrak{x}. \)

Our aim is to produce ``lifts'' \( \tilde{\mathcal{F}}(U) \) of \( \mathcal{F}\res{U} \) to \( \sh^{\cyl,1}_{\Lambda}\res{U} \) for each \( U \in \mathbb{U} \) which, regarded as functors from the \( \infty \)-category \( * \) with one object and contractible endomorphisms, after specifying homotopies over intersections, satisfy the homotopy limit diagram induced by \( \mathbb{U} \) (i.e. the limit over the \v{C}ech nerve of \( \mathbb{U} \) in the category of \( \infty \)-categories).

\subsection{Length maps}\label{sec:bar-lengths}
Throughout we will refer to the ``short'' lift, and ``lengths'' of bars over smooth points, and when checking that lifts over smooth loci are determined by the lift over a point in each component it will be necessary to make precise sense of this. The main statement is the following:
\begin{proposition}\label{prop:bar-length}
Every object \( \mathcal{F} \in \sh^{\bullet,1}_{\Lambda} \) determines a continuous map
\[ l_{\mathcal{F}}: \mathfrak{B} \rightarrow \sym(\R_{\geq 0}). \]
Moreover, over each component \( \mathfrak{U} \) of the smooth locus, \( l_{\mathcal{F}} \) factors through the inclusion \( \sym^{2k}(\R_{>0}) \hookrightarrow \sym(\R_{\geq 0}) \) (given by extension by 0) for some fixed \( k \geq 0 \).
\end{proposition}

\begin{proof}
    We begin by defining \( l_{\mathcal{F}} \) over the smooth locus. In this region \( l_{\mathcal{F}} \) is given via the canonical bar decomposition: over open stratum the bar decomposition is determined by reduction to any point (by stabilization and isotopy). Then at each smooth point in \( \mathfrak{B} \) we define \( l_{\mathcal{F}} \) to be the tuple of lengths of the bars in the decomposition at that point. In the disk case, these are the minimal distances on \( S^1 \) between the paired points, and in the circular case, this is the distance swept out by the support of the bar, read counterclockwise. This is clearly continuous, since the front is smooth over the smooth locus. Moreover any extension satisfies the factorization property by construction.

    It remains to extend over the singular locus. To do this, it suffices to understand the length in a neighborhood of a fold, and a simple crossing. In each case \cref{prop:singular-bar-decomp} shows that the component \( \mathcal{F}_{sm} \) has smooth length spectrum across the singular locus, so it remains only to check \( \mathcal{F}_{sing} \) (by it's construction, \( l_{\mathcal{F}} \) intertwines direct sum of sheaves with the semi-group structure on \( \sym(\R_{\geq 0}) \)).
    
    In the case of a fold \( \mathcal{F}_{sing} \) consists of a sheaf micro-supported on the two strands entering the fold and thus has length converging to 0 along the fold locus. In the crossing case the only potential failure of convergence is if the strands entering the singular locus are paired to each other on one side of the crossing, but are paired to some other strands on the other. This is ruled out by \cref{sec:codim-1}

    All other (generic) singularities appear as limit points of the codimension-1 singular locus, so the extension by limits provides a globally continuous map.
\end{proof}

\begin{definition}\label{def:length-spectrum}
    We call the function \( l_{\mathcal{F}} \) the length spectrum of \( \mathcal{F} \). 
\end{definition}

Note that if \( \mathcal{F} \in \sh^{\C,1}_{\Lambda} \) the length actually takes values in \( \sym([0,\frac{1}{2}]). \) 

Given the length we obtain a function \( L_{\mathcal{F}}: \mathfrak{B} \rightarrow \R_{\geq 0} \) which assigns to each point the maximal length of a bar. This is well defined by the factorization property of \cref{prop:bar-length} and compact support.

\begin{definition}
    We call \( \mathcal{F} \in \sh^{\bullet,1}_{\Lambda} \) \eps-short if \( L_{\mathcal{F}} < \frac{1}{2}-\epsilon \).

    Call a Legendrian \( \Lambda \) disk(cylindrically)-small if there exists \( \epsilon > 0 \) such that every object of \( \sh^{\C,1}_{\Lambda} \) (\( \sh^{\cyl,1}_{\Lambda} \)) is \eps-short. If both are true, call \( \Lambda \) small.
\end{definition}

We would like to relate a more geometric notion of ``size'' of a Legendrian to the shortness of sheaves on it.
\begin{definition}
    Let \( X \) be a smooth Riemannian manifold, and \( Y \) a metric space. A continuous function \( f: X \rightarrow Y \) is called directed-Lipschitz if there exists a continuous function 
    \[ B:STX \rightarrow \R_{\geq 0} \]
    such that
    \[ d\left(f(\exp_t(v)(x)) - f(x)\right) < tB(v) \]
    for any \( t \in \R_{\geq 0}, v\in ST_xX. \)
\end{definition}

The following lemma is evident from our definitions:
\begin{lemma}\label{lem:smooth-maximum}
    Over each component of the smooth locus, the function \( L_{\mathcal{F}} \) is directed-Lipschitz with Lipschitz map bounded by maximal difference between directional derivatives of the strands of the front, regarded as a graph over \( \mathfrak{B} \).
\end{lemma}

As an immediate corollary we get a relationship between the size of a domain into which \( \Lambda \) embeds and the length of bars of a supporting sheaf:
\begin{corollary}\label{lem:small-means-small}
    If \( \Lambda \subset \widehat{D}(1-\epsilon, r_2,\ldots,r_n), \) then \( \Lambda \) is small.
\end{corollary}

\begin{proof}
The slope of the front in the \( x_1 \) direction is bounded between 0 and \( 1-\epsilon \), and \( L_{\mathcal{F}} \) vanishes outside of \( [0,1-\epsilon] \times \R^{n-1} \), so by \cref{lem:smooth-maximum} \( L_{\mathcal{F}}(x_1,\ldots,x_n) \leq (1-\epsilon)\min_{x \in \R\setminus[0,1-\epsilon]} |x_1 - x| < \frac{1}{2} - \frac{\epsilon}{4}. \)
\end{proof}

\subsection{Lifting over simple strata}\label{sec:compatible-pointlike-lifts}
The length map provides a convenient means of checking the isomorphism type of a lift of a disk sheaf obtained from double stopping. In particular, let \( \Xi \subset \widehat{T^*\R^n} \) be a bounded Legendrian with connected, contractible singular locus \( \tilde{\mathbb{X}}_{\Xi} \) having a unique smooth stratum of maximal codimension, \( \tilde{\mathfrak{x}} \). Let \( \mathcal{F} \in \sh^{\C,1}_{\Xi}, \) and suppose \( \mathcal{F} \) is \eps-short for some \eps. Under our assumptions, \( \mathcal{F} \) admits a global singular bar decomposition. Suppose that \( \mathcal{F} \cong \mathcal{F}_{sing} \) (i.e. there are no locally smooth components).

Since the singular locus is contractible, \eps-shortness implies that, near the pre-image of \( \pi(\tilde{\mathfrak{x}}) \), \( p(\Xi) \) is disjoint from an \eps-neighborhood of the antipodal image of \( \tilde{\mathfrak{x}}. \) In particular there exist sections of the bifurcation projection \( \pi:F\setminus p(\Xi) \rightarrow \mathfrak{B}. \) Via \cref{prop:double-stopped-disks} each homotopy class of sections yields a lift of \( \mathcal{F} \) to the cylindrical category. 

Let \( \sigma \) denote such a homotopy class, \( \DS_{\sigma} \) the double stopping functor. Observe that, by the construction of the double stopped lift, for any \( \sigma \) the length spectrum \( l_{\DS_{\sigma}(\mathcal{F})} \) takes vales in \( \sym([0,1)). \) Moreover, if \( \sigma \) denotes the class of the antipode to \( \mathfrak{x}, \)
\[ l_{\DS_{\sigma}(\mathcal{F})} = l_{\mathcal{F}} \]
by \eps-shortness: over points in open strata sufficiently near to \( \tilde{\mathfrak{x}}, \) there exists an interval \( I \subset S^1 \) of length less than \( \frac{\epsilon}{4} \) such that every bar has (at least) one end lying in \( I. \) It follows that the shortest path connecting endpoints of the bars never traverses (the sub-interval of the complement to \( \Xi \) containing) the antipode to \( \tilde{\mathfrak{x}}, \) and so we have the claimed length equality. 

By this argument find that \textit{no other} double stopped lift of \( \mathcal{F} \) has this length spectrum: any other choice of section intersect the short path traversing one of the bars (since it must be separated from the antipode by some components of \( \Lambda \)) and thus the lift of this bar must have length at least \( \frac{1}{2}. \) We have proven the following
\begin{proposition}\label{prop:short-singular-lift}
    Let \( \mathcal{F} = \mathcal{F}_{sing} \) be as above. There is a unique isomorphism class \( \tilde{\mathcal{F}} \in \sh^{\cyl}_{\Xi} \) which can be obtained from \( \mathcal{F} \) via double stopping which satisfies 
    \[ l_{\mathcal{F}} = l_{\tilde{\mathcal{F}}}. \]
\end{proposition}
An additional corollary of the above analysis exactly characterizes the lifts of a \eps-short bar type object:
\begin{corollary}\label{cor:short-bar-lift}
    Suppose \( \Xi \) is a disjoint union of two graphical hypersurfaces, and \( \mathcal{F} \in \sh^{\C,1}_{\Xi} \) (thus a bar type object) is \eps-short. Then \( p(\Xi) \) is disjoint from its antipodal image, and the double stopped lift determined by any hypersurface lying in the component of either (both) antipodal point is short, and the (unique) other lift, given by the other component of the complement, is not.
\end{corollary}

With these in hand we are ready to carry out our lifting procedure, and verify that the resulting objects have compatible isomorphism types. For the remainder of this section, let \( \Lambda \subset \widehat{T^*\R^n} \) be a smooth, closed, front generic Legendrian, and \( \mathcal{F} \in \sh^{\C,1}_{\Lambda} \) a \eps-short microlocal rank-1 sheaf. 

\subsubsection{Open strata}\label{sec:smooth-lifts}
We begin with the open strata of the bifurcation stratification, over each component of which the front \( p(\Lambda) \) is a disjoint union of smooth hypersurfaces, which are graphical over \( \mathfrak{B}. \) Let \( U \in \mathbb{U} \) be an element of our cover which is contained in an open stratum. \( \mathcal{F} \) has a bar decomposition over \( U \) by \cref{thm:disk-pairings} and this has a unique isomorphism class of short circular lift over \( U \) by \cref{cor:short-bar-lift}. Denote this lift by \( \tilde{\mathcal{F}}^{U}. \)

It remains to check that our lifts glue: Let \( \mathfrak{x} \) be an open stratum, \( \set*{U_i} \subset \mathbb{U}_{\mathfrak{x}} \) be a finite subset of \( \mathbb{U} \) such that \( \bigcap U_{i} \neq \emptyset. \) Then the bar decomposition of \( \mathcal{F}\res{\bigcap U_i} \) is canonical, thus isomorphic to that induced by the restrictions from each \( U_i \), and so the restriction of short lifts are all isomorphic to the short lift of \( \mathcal{F}\res{\bigcap U_i}, \) by \cref{cor:short-bar-lift}. This gives the required compatibility over the open strata.

\subsubsection{Pointlike singularities}\label{sec:pointlike-sing}
Next we turn to lifting over the pointlike strata. We induct on the codimension of the strata (with the base case being codimension 0, handled by \cref{sec:smooth-lifts}). To this end, suppose we have constructed, and verified compatibility up to isomorphism of short lifts over the pointlike strata of codimension up to \( n, \) and let \( \mathfrak{x} \) be a smooth codimension \( n \) stratum of the pointlike locus. Let \( U \in \mathbb{U}_{\mathfrak{x}} \) an element of our good cover.

By assumption, \( \mathcal{F}\res{U} \) has a singular bar decomposition 
\[ \mathcal{F}\res{U} \cong \mathcal{F}^U_{sing} \oplus \mathcal{F}^U_{sm}, \]
which has a unique isomorphism type of short double stopped lift in \( \m\sh^{\cyl}_{\Lambda}(U) \) by \cref{prop:short-singular-lift,cor:short-bar-lift}. Let \( \tilde{\mathcal{F}}^{U} \in \m\sh^{\cyl}_{\Lambda}(U) \) denote a lift obtained via antipodal double stopping (component-wise with respect to this decomposition).

Having produced such a lift, there are two types of compatibility we must check. First, there is compatibility with the lifts thus produced over other elements of \( \mathbb{U}_{\mathfrak{x}}, \) which may be verified by an (essentially) verbatim repetition of the argument of \cref{sec:smooth-lifts}. There is another, new, type of compatibility to verify, which is that over intersections of \( U \) with elements of \( \mathbb{U}_{\mathfrak{y}_i} \) for some \( \mathfrak{y}_i \subset \mathbb{X}_{\Lambda} \) strata of smaller codimension. 

Note that, by our inductive hypothesis, it suffices to check for \( \mathfrak{y} \) of codimension \( n-1, \) and by smoothness of the stratification, these may be checked one at a time: for two (distinct) such strata \( \mathfrak{y}_i \), corresponding opens \( U_i \in \mathbb{U}_{\mathfrak{y}_i} \) we necessarily have \( U_1 \cap U_2 = \emptyset. \)

With this in mind, let \( \mathfrak{y} \) be a smooth stratum of codimension \( n-1 \) and \( V \in \mathbb{U}_{\mathfrak{y}} \) a cover element such that \( W := U \cap V \neq \emptyset. \) Passing to a neighborhood of a pointlike front singularity of lower codimension, the singular bar decomposition over \( U \) admits a natural refinement over \( W: \)
\[ \mathcal{F}\res{W} \cong \mathcal{F}^U_{sing}\res{W} \oplus \mathcal{F}^U_{sm}\res{W} \cong \left(\mathcal{F}^U_{sing}\res{W}\right)_{sing} \oplus \left(\mathcal{F}^U_{sing}\res{W}\right)_{sm} \oplus \mathcal{F}^U_{sm}\res{W} \]
by applying the singular bar decomposition over \( W \) to \( \mathcal{F}^U_{sing}. \) On the other hand, by local constancy, the singular bar decomposition of \( \mathcal{F}\res{V} \) is preserved by the restriction to \( W. \)

Now by the characterization of the splitting given in \cref{prop:singular-bar-decomp}, and the isomorphism
\[ \left(\mathcal{F}\res{V}\right)\res{W} \cong \mathcal{F}\res{W} \cong \left(\mathcal{F}\res{U}\right)\res{W} \]
there exists a split isomorphism
\[ \left(\mathcal{F}^U_{sing}\res{W}\right)_{sing} \oplus \left(\left(\mathcal{F}^U_{sing}\res{W}\right)_{sm} \oplus \mathcal{F}^U_{sm}\res{W}\right) \cong \mathcal{F}^W_{sing} \oplus \mathcal{F}^W_{sm}. \]
Since double stopping commutes with restriction this isomorphism lifts to an isomorphism of short lifts by the uniqueness of \cref{prop:short-singular-lift}.

This verifies the required compatibility for lifting over the pointlike singularities of \( \Lambda. \)

\subsection{Splittings and homotopies}
Let \( \mathcal{C} \subset \mathcal{D} \) be \( \infty \)-categories, such that the inclusion is a homotopy equivalence. Let \( \iota:\mathcal{C}\rightarrow \mathcal{D} \) denote this inclusion, and let \( \id:\mathcal{D} \rightarrow \mathcal{D} \) denote the identity functor. Define the \textit{category of splittings}
\[ Sp(\mathcal{D},\mathcal{C}) := \left( \id \sim \iota \right) \]
as the iso-comma category of the identity over this inclusion. Each of these projections are acyclic Kan fibrations, and thus admit sections, the space of which is contractible.

In more down-to-earth language, \( Sp(\mathcal{D},\mathcal{C}) \) can be regarded as the category whose objects are triples \( (d,c,f) \) where \( d \in \mathcal{D}, \) \( c \in \mathcal{C}, \) and \( f:d \rightarrow c \) is an isomorphism, and whose morphisms are commuting diagrams
\[
    \begin{tikzcd}
        d \arrow{r}{f} \arrow{d}{g} & c \arrow{d}{g'}\\
        d' \arrow{r}{f'} & c'
    \end{tikzcd}.
\]
Indeed, this is the corresponding construction on the level of homotopy categories, and has the desired properties in our setting. One can check (see, e.g. \cite{HTT}) that the \( \infty \)-categorical incarnation described above simply provides a lift of these to the homotopy-coherent world.

In our setting, we have sheaves of \( \infty \)-categories \( \m\sh^{\C},\m\sh^{\cyl} \) each on the open site of \( \mathfrak{B}. \) The results of \cref{sec:combinatorics} present \( \m\sh^{\bullet}(U) \) as (quasi-)equivalent to the inclusion of a full subcategory over a base of contractible opens. The above argument then yields an isomorphic sheaf of \( \infty \)-categories taking values in the corresponding categories of splittings over the elements of our base. In particular, every object of the global sections can be equipped with a canonical (up to contractible choice) family of local splittings which are compatible with restriction. 

This presents every disk object as one which is canonically split into (singular) bar-type objects, equipped with canonical (higher) homotopies between restrictions of splittings. This is precisely the data necessary for our lifting techniques, which we have seen in the previous sections yield canonical isomorphism classes short lifts when our object is globally short. In order to produce a global object by lifting this way, one must verify that these lifts can be assembled into homotopy coherent descent data (i.e. a functor from the \v{C}ech nerve of our good cover), and thus that splittings into bars can be so prescribed. The argument above shows that, before lifting, we may present any object in such a way.

In fact, given an object of \( \sh^{\C}_{\Lambda} \) presented as homotopy coherent descent data equipped with splittings, we may further assume that our homotopies respect splittings: Morphisms of split objects are themselves split by pairs of components, and isomorphisms are exactly those morphisms whose restriction to each endomorphism component are themselves isomorphisms. Now given two bars (or a bar and singular bar) \( b,b' \) with disjoint microsupport, it follows from the quiver presentation of \cref{sec:combinatorics} that any degree 0 morphism \( b \rightarrow b \) which factors through \( b' \) is trivial (there are non-trivial degree morphisms of simple such bars in at most one direction). This weak orthogonality implies that, given coherent split descent data, we may simply project out the non-endomorphism components of our homotopies and obtain coherent descent data for the same object which respects the splitting data. Said another way, we can always modify the splitting morphisms \( f \) (in the notation used above) such that morphisms on the split side have no cross terms.

Equipped with this data, we can check that these homotopies lift according to the procedure of the preceding sections: A \(n\)-homotopy is the data of an element \( \eta \) of \( \hom^{n+1}(\mathcal{F},\mathcal{G}) \) realizing the appropriate (signed) linear combination (arising from some simplicial presentation) of \( (n-1) \)-homotopies as the image under boundary of \( \eta \) (the simplest case is when the relevant object is a 2-simplex, when we recover the usual notion of chain homotopy). In light of the previous paragraph, it suffices to check that the endomorphism algebras of bars and singular bars lift, which is immediate by \cref{prop:double-stopped-disks}. Thus, starting with an object presented as nice split descent data as above, the lifting constructions of the previous two subsections suffice to obtain coherence up to all higher homotopies over the pointlike locus, and it remains only to verify that this may be extended over the XX locus. We check this in the next subsection.

\subsection{Lifting the XX locus}\label{sec:XX-sing}
It remains to produce circular lifts over the XX locus of the bifurcation stratification. Again we will induct on the codimension of the strata. Recall that the XX-locus has minimal codimension 2. 

The setting is as follows: Let \( \mathfrak{x} \) be a smooth XX-type type stratum of \( \mathbb{X}_{\Lambda}, \) \( U \in \mathbb{U}_{\mathfrak{x}}. \) By assumption, \( \mathfrak{x} \) is the intersection of the image of some finite collection \( \tilde{\mathfrak{x}}_i \) of smooth strata of the frontal stratification \( \tilde{\mathbb{X}}_{\Lambda} \) under the front-to-bifurcation projection \( \pi. \) In particular, over \( U \) the singular locus \( \tilde{\mathbb{X}}_{\Lambda}\cap U \) may be assumed to be a finite disjoint union of contractible smoothly stratified subspaces.

By \cref{lem:stratum-local-constancy} \( \m\sh^{\bullet}_{\Lambda}(U) \) is computed (via the canonical equivalence induced by reduction) by \( \m\sh^{\bullet}_{\Lambda}\res{\mathfrak{S}} \) where \( \mathfrak{S} \subset U \) is a smooth, relatively clopen disk transverse to \( \mathbb{X}_{\Lambda} \) and intersecting \( \mathfrak{x} \) at a point \( x. \) More precisely, \( \Lambda \cap U \) is isotopic to a stabilization of \( \Lambda\res{\mathfrak{S}} \)\footnote{By this we denote the image of \( \Lambda \) under reduction along the pre-image of \( \mathfrak{S} \) along \( \Pi. \)}.

In fact, since \( \mathfrak{x} \) is of XX-type, \( \m\sh^{\bullet}_{\Lambda}(U) \) can be computed in terms of an even more refined subset: Let \( L_x := \partial\mathfrak{S} \) denote the boundary sphere to our slice (this is the link of our XX-singularity in the slice). Let \( H \subset T_{x}\mathfrak{S} \) be a hyperplane transverse to the strata in the star of \( \mathfrak{x}, \) and \( \mathfrak{H} \subset \mathfrak{S} \) a smoothly embedded relatively clopen disk transverse to the star, passing through \( x \) with tangent plane \( H. \) Such a disk exists locally by exponentiation and can be extended over \( \mathfrak{S} \) by local triviality of the front over \( U. \) 

Such \( \partial\mathfrak{H} \) bisects \( L_x \) into a pair of hemispheres, \( L^{\pm}_x, \) each isotopic relative \( \partial\mathfrak{H} \) to \( \mathfrak{H} \) via hypersurfaces transverse to \( \mathbb{X}_{\Lambda} \cap \mathfrak{S}. \) In particular, \( \Lambda\mid_{\mathfrak{S}} \) can be regarded as the trace of an isotopy from \( \Lambda\res{L^+_x} \) to \( \Lambda\res{L^-_x} \) (or vice-versa). In such a situation, (a special case of\footnote{In fact, this isotopy is particularly trivial, in the sense that it is actually induced by an isotopy of the front by its construction. In this setting one could prove an equivalence by hand using (frontal) diffeomorphism invariance, but the more sophisticated statement of \cite{GKS12} is more convenient for the manipulation we wish to make.}) the invariance of microsheaves (as a sheaf of categories) under contactomorphism (\cite{NadlerShende-weinstein}), together with the results of \cite{GKS12} establish canonical equivalences between the sheaf categories over \( L^+_x,L^-_x \) and \( \mathfrak{S}. \) More precisely we have (in the notation of the preceding paragraphs):
\begin{proposition}[\cite{GKS12,NadlerShende-weinstein}]\label{prop:XX-is-isotopy}
    Fix an isotopy \( \phi_t \) realizing the movie induced by \( \Lambda\res{\mathfrak{S}}. \) Then there are canonical equivalences
    \[ \m\sh^{\bullet}_{\Lambda}\res{L^-_x}(L^-_X) \leftarrow \m\sh^{\bullet}_{\Lambda}\res{\mathfrak{S}}(\mathfrak{S}) \rightarrow \m\sh^{\bullet}_{\Lambda}\res{L^+_x}(L^+_x). \]
    Moreover, the pushforward
    \[ \phi_*:\m\sh^{\bullet}_{\Lambda}\res{L^+_x} \rightarrow \m\sh^{\bullet}_{\Lambda}\res{L^-_x} \]
    is an isomorphism of sheaves of (\( \infty \)-)categories, inducing the above equivalence on global sections.
\end{proposition}
\begin{remark}
    We have been invoking exactly this result of \cite{GKS12} throughout when employing operations such as stabilization, or straightening via isotopy, but we emphasize it's presence here since this situation is a rather more non-trivial isotopy than we have been accustomed to dealing with thus-far in that it is not fibered over the bifurcation space.
\end{remark}
Note that this result implies that there is a unique extension of a sheaf defined over \( L^+_x \) to all of \( \mathfrak{S} \) up to isomorphism.

Now, in the setup above suppose \( \mathcal{F} \in \sh^{\C,1}_{\Lambda}(U) \). Let \( \tilde{\mathcal{F}}^{\circ} \) be the short circular lift of \( \mathcal{F}\res{\mathfrak{S} \setminus \set*{x}} \) furnished by our inductive hypothesis and the preceding section. 
\begin{proposition}\label{prop:XX-extension}
Let \( L^{\pm}_x,\mathfrak{H} \) be as above. Let \( \tilde{\mathcal{F}}^{\circ,+} \) denote the restriction to \( L^+_x, \) and let \( \mathcal{G} \) denote the object over \( \mathfrak{S} \) induced by the equivalence of \cref{prop:XX-is-isotopy}. Then
\[ \mathcal{G}\res{\mathfrak{S} \setminus \set*{x}} \cong \tilde{\mathcal{F}}^{\circ} \]
In particular, \( \tilde{\mathcal{F}}^{\circ} \) can be extended over \( x \).
\end{proposition}

\begin{proof}
    Fix an isotopy \( \phi_t \) realizing \( \Lambda\res{\mathfrak{S}} \) as in \cref{prop:XX-is-isotopy}. Note that the pointlike strata of \( \mathbb{X}_{\Lambda} \cap L^+_x \) are naturally identified with those of \( \mathbb{X}_{\Lambda}\cap L^-_x \) via their lifts to \( \tilde{\mathbb{X}}_{\Lambda}, \) and \( \phi_t \) necessarily realizes this identification. 

    Let \( \mathfrak{x} \) such a stratum of over \( L^+_x, \) and \( \mathfrak{y} \) the corresponding stratum over \( L^-_x \). Then by uniqueness \( \phi_* \) carries a singular bar decomposition of \( \mathcal{F} \) over a contractible open \( V_{\mathfrak{x}} \) over \( \mathfrak{x} \) onto a singular bar decomposition of \( \mathcal{F} \) over the corresponding \( V_{\mathfrak{y}} \) over \( \mathfrak{y}. \) It follows that 
    \[ \tilde{\mathcal{F}}^{\circ}\res{U_{\mathfrak{y}}} \cong \mathcal{G}\res{U_{\mathfrak{y}}} \]
    by comparing the corresponding pushforwards after double stopping. Homotopies trivializing restriction similarly pushforward on the level of disk categories, and thus so do their lifts. 

    Over nearby XX-strata of lower codimension, we have obtained our lift by the same procedure; in particular the pushforward restricts over slices to such strata, and the isomorphism induced by pushforward verifies the required compatibility of lifts. 
\end{proof}
This yields lifts over every element of \( \mathbb{U}, \) for which compatibility (up to isomorphism class) may be verified analogously to \cref{sec:compatible-pointlike-lifts}. Higher homotopies follow similarly: we may choose the isotopy by which we propagate across the XX locus coherently over each component, then the restriction away from the XX locus is handled by the arguments for pointlike singularities above. Thus we've produced globally coherent lifts, and so completed the proof of \cref{thm:pushing-out}.

\section{Legendrian non-squeezing}\label{sec:non-squeezing}
In this section we'll show that the lifts of an assortment of familiar Lagrangians in \( \C^n \) are not squeezable. In particular, we compute the categories \( \sh^{\C}_{\Lambda} \) and \( \sh^{\cyl}_{\Lambda}. \)

\subsection{The unknot}\label{sec:unknots}
The fundamental example is the non-trivial unknot considered in \cref{sec:3d-nonsqueezing}, whose front is drawn there.

The front divides the cylinder into two unbounded pieces, so it follows immediately that \( \sh^{\cyl}_{\Lambda} \) is trivial. On the other hand, after embedding into the standard open book, \( \Lambda \) is Legendrian isotopic to a local unknot (by passing a smooth strand through the binding), and thus supports a non-trivial microsheaf. In particular, \( \sh^{\C}_{\Lambda} \simeq \vect_{\kk_2}, \) the (2-periodic) derived category of finite dimensional vector spaces. Under this equivalence (given by the microstalk) the microlocal-rank-1 objects are exactly the vector spaces of rank 1, so there are 2 non-isomorphic objects of \( \sh^{\C,1}_{\Lambda}. \) Thus, as a consequence of \cref{thm:pushing-out}, \( \Lambda \) is non-squeezable. 

\subsection{\(T^2_{cl}(1,1)\)}\label{sec:2d-clifford}
In the case \( \Lambda \) is a Legendrian lift of the monotone Clifford torus of dimension 2, bounding holomorphic disks of minimal area 1, we can also compute microsheaves by hand. The front of this lift is the spin of the front:
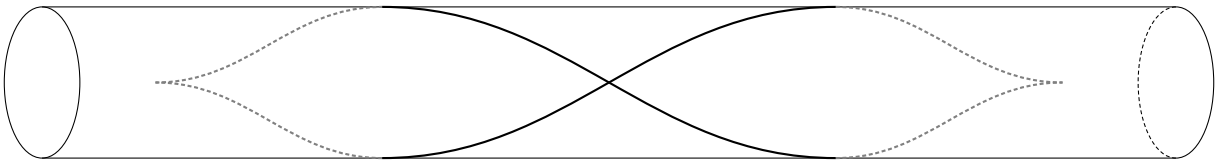
\begin{center}
   \begin{tikzpicture}[mydash/.style={dashed,dash pattern=on 1.5pt off 1pt}]
    \draw (0,0) ellipse (0.5 and 1);
    \draw (15,-1) arc(-90:90:0.5 and 1);
    \draw[mydash] (15,1) arc(90:270:0.5 and 1);
    \draw (0,-1) -- (15,-1);
    \draw (0,1) -- (15,1);
    \draw[thick,mydash,gray] (1.5,0) to[out=0,in=180] (4.5,1);
    \draw[thick,mydash,gray] (4.5,-1) to[out=180, in=0] (1.5,0);
    \draw[thick] (4.5,1) to[out=0,in=180] (10.5,-1);
    \draw[thick] (4.5,-1) to[out=0,in=180] (10.5,1);
    \draw[thick,mydash,gray] (10.5,1) to[out=0,in=180] (13.5,0);
    \draw[thick,mydash,gray] (10.5,-1) to[out=0,in=180] (13.5,0);
   \end{tikzpicture}
\end{center}
About the central intersection fiber.

In the case of \( \sh^{\C,1}_{\Lambda} \) we can apply the spin of the isotopy from the non-trivial unknot to a local one, and obtain the usual front for the Clifford Legendrian induced by the central toric fiber of \( \C\mathbb{P}^2, \) whose isomorphism classes of objects over \( \kk_2 \) are in bijection with (shifts of) elements of \( \kk^{\times} \setminus \set*{1} \) (cf. \cite{Casals-Zaslow}*{\S 6.1.1}).

On the other hand, in the case of \( \sh^{\cyl,1}_{\Lambda} \) we can resolve the conical singularity at the axis of the spinning to the following local bifurcation
 picture (see \cite{Rizell-surfaces}*{\S 3} for a detailed construction):
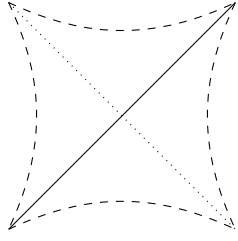
\begin{center}
   \begin{tikzpicture}
        \draw[dashed] (0,0) to [bend right=25] (3,0);
        \draw[dashed] (0,0) to [bend left=25] (0,-3);
        \draw[dashed] (3,0) to [bend right=25] (3,-3);
        \draw[dashed] (0,-3) to [bend left=25] (3,-3);
        \draw[dotted] (0,0) to (3,-3);
        \draw (0,-3) to (3,0);
   \end{tikzpicture}
\end{center}

Here the dashed lines denote cusps, and the solid and dotted lines denote crossings, one thought of as above, and one below the envelope enclosed by this the front over this region.

Then by the cusp conditions, every sheaf microsupported on this front simply pairs the strands of the front in the unique way away from the resolved singularity, and it is simply a matter of determining how to extend this over the central region, in particular, what happens at the crossings. Computing directly, one finds that there is a unique (up to isomorphism) such extension over any finite field, and thus exactly one isomorphism class of object of \( \sh^{\cyl,1}_{\Lambda} \) in this case. 

We thus conclude that the lift of the monotone Clifford Lagrangian of dimension 2, and area 1 is not squeezable.


\subsection{Bifurcation suspensions}\label{sec:bif-suspensions}
Let us now describe a general class of Legendrians which possess fronts particularly amenable to the calculation of microsheaves. 

Let \( \Lambda \subset \widehat{T^*\R^{k}} \) be a closed, smooth Legendrian, \( M^{n-k} \) a closed, smooth manifold, and \( \phi:M \rightarrow \R^{n} \) a smooth embedding with trivial normal bundle. Fix a trivialization \( \nu(\phi) \cong M \times \R^k, \) and a tubular neighborhood \( \psi:M \times B^k(r) \rightarrow \R^{n} \) of \( \phi(M). \)

Identifying the bifurcation space over which \( \Lambda \) lives with \( \R^k \), by rescaling we may assume that \( \Pi(\Lambda) \) lands in \( B^k(r). \) Then \( \psi \times \id_{S^1} \) induces a front for a closed, smooth Legendrian submanifold of \( \widehat{T^*\R^n} \) which is diffeomorphic to \( M \times \Lambda. \) We denote this new Legendrian by \( \Sigma_{\phi}\Lambda. \)

In fact, the above construction yields a family of (not necessarily non-isotopic) Legendrians indexed by \( \alpha \in \hom(\pi_1(M;\Z),\Z) \) by rotating the \( S^1 \) factor of the front the specified number of times along a given loop. We denote the Legendrian obtained this way by \( \Sigma^{\alpha}_{\phi}(\Lambda), \) with the convention that \( \alpha = 0 \) when the superscript is omitted.

The advertised feature of this construction is relative ease in computation of microsheaves of its output:
\begin{theorem}\label{thm:suspended-microsheaves}
    \( \mush^{\bullet}_{\Sigma^{\alpha}_{\phi}(\Lambda)} \) is (quasi-)equivalent to the category of locally constant sheaves on \( M \) valued in \( \sh^{\bullet}_{\Lambda}. \)
\end{theorem}

\begin{proof}
    By our support conditions (vanishing stalk outside the bifurcation image of \( \Lambda \)) 
    \[ \mush^{\bullet}_{\Sigma^{\alpha}_{\phi}(\Lambda)} \cong \psi_*\mush^{\bullet,\nu M}_{\Sigma^{\alpha}_{\phi}(\Lambda)}, \]
    where the latter sheaf of categories is microsheaves on the appropriate filling of the pre-quantization of the normal bundle to \( M \) with microsupport on the natural Legendrian. Since \( \psi \) is a smooth embedding, and microsheaves is compactly supported, \( \psi_* \) is fully faithful, and it suffices to compute \( \mush^{\bullet,\nu M}_{\Sigma^{\alpha}_{\phi}(\Lambda)}. \) 

    Pushing down along the projection \( \nu M \rightarrow M, \) we find, reducing along fibers of the normal bundle and applying local triviality of the suspension, that this is the constant (\( \infty \)-)sheaf on \( M \) with stalk \( \sh^{\bullet}_{\Lambda} \) (global sections of microsheaves on \( \Lambda \).) By constancy, and the definition of sheaf, global sections are thus computed by functors from the \v{C}ech nerve of a good cover to \( \sh^{\bullet}_{\Lambda}, \) which are naturally identified with the locally constant sheaves on \( M \) by extending from the cover to a base for the topology of \( M \) by isomorphisms.
\end{proof}

\subsubsection{Whitney spheres \& twist tori}\label{sec:twist-torus-fronts}
The above construction can also be understood symplectically: If a closed Legendrian \( \Lambda \subset \widehat{T^*\R^n} \) has pre-quantization projection an embedded Lagrangian, then any bifurcation suspension of \( \Lambda \) does as well. Moreover, if \( \Lambda \) is monotone, then so is \( \Sigma_{\phi}(\Lambda). \) Applying this to a lift of \( T^{n}_{cl}(1,\ldots,1) \) along a standard \( S^1 \rightarrow \R^2 \hookrightarrow \R^{n+1} \) it is not hard to check that one obtains a Legendrian lift of the twist torus (in the sense of \cite{Chekanov-Schlenk}) corresponding to our original Clifford torus. Indeed, the original construction of \cite{Chekanov-tori} presents Chekanov's monotone twist tori as exactly such suspensions along embedded copies of neighborhoods of the 0-section in \( T^*S^1. \) Iterating, one obtains all such ``twist tori'' in this way.

Similarly, the unknot of \cref{sec:unknots} suspended along the standard embedding of the unit sphere yields a Legendrian lift of the Maslov-2 resolution of the Lagrangian Whitney sphere of area 1 (in dimension 2 this recovers the Chekanov torus). 

By \cref{thm:suspended-microsheaves} the microlocal rank 1 objects are exactly local systems in the microlocal rank 1 objects of \( \sh^{\bullet}_{\Lambda}. \) When \( \Lambda \) is the unknot of \cref{sec:unknots} there is exactly one object of \( \sh^{\C,1}_{\Lambda} \), and \( \sh^{\cyl,1}_{\Lambda} \) is empty. Thus there are no microlocal-rank-1 cylindrical local systems for any suspension, and there is always at least one disk object provided by the constant sheaf. The non-monotone Clifford tori mentioned in \cref{thm:high-d-squeezing} are just suspensions of the non-trivial unknot along non-trivial \( \alpha \in H^1(M;\Z) \), and so also follow by the same computation. Similarly, when \( \Lambda \) is a lift of \( T^2_{cl}(1,1), \) \( \sh^{\cyl,1}_{\Lambda} \) is simply rank-1 local systems in \( \vect_{\kk} \) on the base of the suspension, while \( \sh^{\C,1}_{\Lambda} \) is local systems valued in \( \sh^{\C,1}_{T^2_{cl}}. \) In particular, over \( \F_p \) the former category has cardinality bounded by the cardinality of the cohomology of \( M \) with \( \F_p \) coefficients, while the latter has (for large \(p\)) of order \( p \) times the cardinality of \( H^{\bullet}(M;\F_p) \) many isomorphism classes, again violating the inequality of \cref{thm:pushing-out}. Applying this to twists of the 1-dimensional and 2-dimensional Clifford tori, and resolutions of the Whitney sphere, we conclude \cref{thm:high-d-squeezing} in the case \( A = 1. \) The general case \( A \in \N \) follows by passing to an \( n \)-fold cover of the pre-quantization.

\bibliographystyle{alpha}
\bibliography{legendrian_squeezing}{}
\end{document}